\documentclass[11pt]{article}

\usepackage{hyperref}
\hypersetup{
    unicode      = {false},
    pdftoolbar   = {true},
    pdfmenubar   = {true},
    pdffitwindow = {true},
    pdfauthor    = {Audet et al.},
    pdftitle     = {catmads},
    pdfsubject   = {catmads},
    pdfkeywords  = {catmads},
    pdfnewwindow = {true},
    colorlinks   = {true},
    linkcolor    = {blue},
    citecolor    = {blue},
    filecolor    = {black},
    urlcolor     = {blue},
    breaklinks   = {true}
}

\usepackage[margin=1in]{geometry}
\usepackage{booktabs}
\usepackage[utf8]{inputenc}
\usepackage{comment}
\usepackage[english]{babel}
\usepackage[T1]{fontenc}
\usepackage[export]{adjustbox}
\usepackage{booktabs}
\usepackage{multirow}
\usepackage{bm}
\usepackage{graphicx}
\usepackage{epstopdf}
\usepackage{soul}

\usepackage{amsmath}
\DeclareMathOperator*{\argmin}{argmin}

\usepackage{stackengine}
\usepackage{mathtools}
\usepackage{stackrel}
\usepackage{dsfont}

\usepackage{mathrsfs}   
\usepackage{changepage}   
\usepackage{enumitem}

\usepackage{amsfonts}
\usepackage{makecell}
\usepackage{amssymb}
\usepackage{lipsum}
\usepackage{multicol}
\usepackage[font=small]{caption}
\usepackage{subcaption}
\usepackage{float}
\usepackage[separate-uncertainty=true]{siunitx}
\usepackage{tikz}

\usepackage{forest}
\usepackage{stanli}
\usetikzlibrary{positioning,shapes,shadows,arrows ,automata,arrows.meta,decorations, decorations.text,calligraphy}
\usetikzlibrary{matrix}
\usepackage{pgfplots}
\pgfplotsset{compat=1.18}
\usetikzlibrary{hobby, decorations.markings, arrows.meta}

\usepackage{xfrac}
\usepackage{standalone}
\usepackage{yhmath}

\makeatletter \@mparswitchfalse \makeatother
\normalmarginpar 
\usepackage[textwidth=20mm,backgroundcolor=yellow,linecolor=orange,textsize=scriptsize]{todonotes}

\usepackage{arydshln}

\definecolor{Red}{rgb}{1,0,0}
\definecolor{Green}{rgb}{0,.6,0}
\definecolor{Blue}{rgb}{0,0,1}

\newcommand{\x}{\bm{x}}
\newcommand{\y}{\bm{y}}
\newcommand{\z}{\bm{z}}
\newcommand{\w}{\bm{w}}

\newcommand{\quant}{\mathrm{qnt}}
\newcommand{\integer}{\mathrm{int}} 
\newcommand{\continuous}{\mathrm{cont}} 
\newcommand{\cat}{\mathrm{cat}}

\newcommand{\feasible}{\text{\textsc{fea}}}
\newcommand{\infeasible}{\text{\textsc{inf}}}

\newcommand{\diag}{\mathrm{diag}}

\usepackage{fancyhdr}
\usepackage{csquotes}
\usepackage{xcolor}

\usepackage[algo2e]{algorithm2e} 
\DontPrintSemicolon

\usepackage{cleveref}
\Crefname{subsection}{Section}{Sections}

\crefname{appendix}{}{}  

\newcommand{\mads}{{\sf MADS}\xspace}
\newcommand{\catmads}{{\sf CatMADS}\xspace}
\newcommand{\gmads}{{\sf G-MADS}\xspace}
\newcommand{\mvmads}{{\sf MV-MADS}\xspace}
\newcommand{\nomad}{{\sf NOMAD}\xspace}

\usepackage{algorithm}
\usepackage{algorithmic}

\usepackage{amsthm}
\theoremstyle{plain}
\newtheorem{definition}{Definition}
\newtheorem{theorem}{Theorem}

\title{{\catmads: Mesh Adaptive Direct Search for constrained blackbox optimization with categorical variables}}

\author{\href{https://www.gerad.ca/Charles.Audet/}{Charles Audet}\thanks{GERAD and Department of Mathematics and Industrial Engineering, Polytechnique Montr\'eal. 6079, Succ. Centre-ville Montr\'eal, Qu\'ebec H3C 3A7, Canada (\href{mailto:Charles.Audet@gerad.ca}{Charles.Audet@gerad.ca},
\href{mailto:youssef.diouane@polymtl.ca}{youssef.diouane@polymtl.ca},
\href{mailto:edward.halle-hannan@polymtl.ca}{edward.halle-hannan@polymtl.ca}, 
\href{mailto:sebastien.le-digabel@polymtl.ca}{sebastien.le-digabel@polymtl.ca},
\href{mailto:christophe.tribes@polymtl.ca}{christophe.tribes@polymtl.ca}) \\%
\hspace*{1.4em}$^\dagger$First and corresponding author.}
    \and
    \href{https://www.gerad.ca/en/people/youssef-diouane}{Youssef Diouane}\footnotemark[1]
    \and
    \href{https://www.gerad.ca/en/people/edward-halle-hannan}{Edward Hall\'e-Hannan}$^{\dagger,}$\footnotemark[1]
    \and
    \href{https://www.gerad.ca/Sebastien.Le.Digabel/}{S\'ebastien Le~Digabel}\footnotemark[1]
    \and
    \href{https://www.gerad.ca/en/people/christophe-tribes}
    {Christophe Tribes}\footnotemark[1]
}

\begin{document}

\maketitle

\begin{center}
    \text{\Large 
    \textbf{Abstract}}
\end{center}

\begin{adjustwidth}{30pt}{30pt}

Solving optimization problems in which functions are blackboxes and variables involve different types poses significant theoretical and algorithmic challenges.
Nevertheless, such settings frequently occur in simulation-based engineering design and machine learning.
This paper extends the Mesh Adaptive Direct Search (\mads) algorithm to address mixed-variable problems with categorical, integer and continuous variables.
\mads  is a robust derivative-free optimization framework with a well-established convergence analysis for constrained quantitative problems.
%
\catmads generalizes \mads  by incorporating categorical variables through distance-induced neighborhoods.
%
%
A detailed convergence analysis of \catmads is provided, with flexible choices balancing computational cost and local optimality strength.
Four types of mixed-variable local minima are introduced, corresponding to progressively stronger notions of local optimality.
\catmads integrates the progressive barrier strategy for handling constraints, and ensures Clarke stationarity.
%
%
%
An instance of \catmads employs cross-validation to construct problem-specific categorical distances.
This instance is compared to state-of-the-art solvers on 32 mixed-variable problems, half of which are constrained.
Data profiles show that \catmads achieves the best results, demonstrating that the framework is empirically efficient in addition to having strong theoretical foundations. \\


\end{adjustwidth}

\noindent
{\small
\textbf{Funding:} This research is funded by a Natural Sciences and Engineering Research Council of Canada (NSERC) PhD Excellence Scholarship (PGS D), a Fonds de Recherche du Qu\'ebec (FRQNT) PhD Excellence Scholarship and an Institut de l’\'Energie Trottier (IET) PhD Excellence Scholarship, 
as well as by the NSERC discovery grants  
    RGPIN-2020-04448 (Audet),
    RGPIN-2024-05093 (Diouane)
    and RGPIN-2024-05086 (Le~Digabel).
}

%


\section{Introduction} 
\label{sec:intro}

Many engineering and machine learning problems involve functions without explicit expressions, known as \textit{blackbox} functions.
These functions often include \textit{categorical} variables that take qualitative values.
In this context, traditional gradient-based or relaxation-based methods can not be directly employed.
%
This work proposes \catmads,
an extension of the Mesh Adaptive Direct Search (\mads) algorithm~\cite{AuDe2006},
that addresses constrained blackbox problems with quantitative and categorical variables.
%
%
\catmads manages categorical variables with problem-specific distances.
%
%

%


This work covers many real-life applications.
Machine learning (ML) models are characterized by hyperparameters that can be optimized to improve performance.
Hyperparameter tuning, involves optimizing a time-consuming mixed-variable blackbox that trains the model and returns a test score~\cite{G-2022-11, G-2024-33}.
%
In~\cite{solar_paper}, a simulation estimates the energy produced by a solar powerplant for given design variables, such as dimensions of pieces or turbine types.
%
%
An active research field in engineering is the simulation-based design of aircraft~\cite{BuCiDeNaLa2021, SaBaDiLeMoDaNgDe2021}.
%
%
These simulation input geometry and materials, and output energy-based scores via numerical methods~\cite{MaLa2013}.

\subsection*{Notation} Vectors are in bold font.
The letters $\x$ and $\y$ are reserved for points.
Scalars are in normal font.
Subscripts without parenthesis denote variables, \textit{e.g.}, $x_1^{\cat}$ is the first categorical variable.
Subscripts with parentheses are reserved for indexing iterations, \textit{e.g.}, $\x_{(k)}$, or for listing points, \textit{e.g.}, $\x_{(1)}, \x_{(2)}, \ldots, \x_{(s)}$.
For functions, subscripts represent the parameters on which they depend.
Superscripts are reserved for types, with small exceptions in the convergence analysis.
$\mathbb{R}$, $\mathbb{Z}$ and $\mathbb{N}$ denote the set of real numbers, integers and nonnegative integers, respectively.
Other sets denoted using capital letters with either normal or calligraphic font, \textit{e.g.}, A or $\mathcal{A}$, except for the set of known points noted $\mathbb{X}$.
For a neighborhood, the semicolon distinguishes its center from the parameter controlling its size, \textit{e.g}, an open ball of radius $\epsilon$ centered at $\mu$ is noted $\mathcal{B}(\bm{u};\epsilon)$.

\subsection{Scope}
\label{sec:problem_statement}

This work considers inequality constrained mixed-variable blackbox optimization problems of the form  
\begin{equation}
    \min_{ \x \in \Omega } f(\x), 
    \label{eq:formulation}
\end{equation}
where $f:\mathcal{X} \to \overline{\mathbb{R}}$ (with $\overline{\mathbb{R}} = \mathbb{R} \cup \{+ \infty \}$) is the objective function, $\mathcal{X}$ is the domain of the objective and constraint functions,
%
%
$\Omega \coloneq \{ \x \in \mathcal{X} \ : \ g_j( \x) \leq 0, \  j \in J \}$ is the feasible set, $g_j:\mathcal{X} \to \overline{\mathbb{R}}$ is the $j$-th constraint function of the problem, with $j \in J$, and $|J|$ denotes the number of constraints.
%

The objective and the constraint functions are assumed to be provided by blackboxes, defined as in~\cite{AuHa2017}:~``\textit{any process that when provided an input, returns an output, but the inner working of the process is not analytically available}''.
%
In this context, derivatives are unavailable and the function evaluations are performed by executing a potentially time-consuming process.
For example, the estimation of an aircraft drag could be done using a finite element simulation that outputs a drag value, provided a geometry and materials used.
%

The execution of a blackbox can crash or fail to return a valid value because of \textit{hidden constraints}~\cite{ChKe00a}, \textit{e.g.}, an unexpected division by zero.
In that case, the point $\x \in \mathcal{X}$ is assigned $f(\x)=+\infty$.
A point $\x \not \in \mathcal{X}$ can also be assigned $f(\x)=+\infty$ to outline that an \textit{unrelaxable} constraint, which must be satisfied to have a proper blackbox execution, is not satisfied. 
In contrast, the constraint functions characterizing the feasible set $\Omega$ are said to be \textit{relaxable and quantifiable}, since they can be evaluated and return a proper value without being satisfied~\cite{LedWild2015}.

The target optimization problems are mixed-variable, meaning that a variable can be of any type among \textit{continuous} ($\continuous$), \textit{integer} ($\integer$) and \textit{categorical} ($\cat$).
%
Variables of the same type are regrouped into a column vector called a \textit{component} and expressed as
\begin{equation}
    \x^t \coloneq (x^t_1, x^t_2, \ldots, x^t_{n^t}) \in \mathcal{X}^t \coloneq \mathcal{X}^t_1 \times \mathcal{X}^t_2 \times \ldots \times \mathcal{X}^t_{n^t},
    \label{eq:component}
\end{equation}
where $t \in \{ \cat, \integer, \continuous  \}$ refers to the type, $n^t \in \mathbb{N}$ is the number of variables of type $t$, $\mathcal{X}^t$ is the set of type $t$, $x_i^t \in \mathcal{X}_i^t$ denotes the $i$-th variable of type $t$, and $\mathcal{X}_i^t$ contains the bounds of $x_i^t$.  
For convenience, the indices of the variables of type $t$ are regrouped in the set $I^t \coloneq \{1,2,\ldots,n^t\}$.

The continuous $\x^{\continuous} \in \mathcal{X}^{\continuous} \subseteq \mathbb{R}^{n^{\continuous}}$ and integer $\x^{\integer} \in \mathcal{X}^{\integer} \subseteq \mathbb{Z}^{n^{\integer}}$ components are composed of \textit{quantitative} variables that belong to ordered sets with well-defined metrics.
For readability, the integer and continuous components can be grouped into the quantitative component $\x^{\quant} \coloneq \left( \x^{\integer}, \x^{\continuous} \right) \in \mathcal{X}^{\quant} \coloneq \mathcal{X}^{\integer} \times \mathcal{X}^{\continuous}$.
%
%
The integer $\x^{\integer}$ and categorical $\x^{\cat}$ components both contain discrete variables.
However, categorical variables are not quantitative and take qualitative values, known as \textit{categories}.
%
For $i \in I^{\cat}$, the $i$-th categorical variable $x_i^{\cat} \in \mathcal{X}_i^{\cat} \coloneq \{c_1, c_2, \ldots, c_{\ell_i} \}$, where $\ell_i \in \mathbb{N}$ is the number of categories and $c_j$ is a category for $j \in \{1,2,\ldots, \ell_i\}$, $i \in I^{\cat}$.
For example, the blood type of a person takes a category in the set $\{\text{O-},\text{O+},\ldots,\text{AB+}\}$. 
%
The usual distance functions are not suitable for categorical variables, which makes them difficult to treat~\cite{G-2022-11, hastie01statisticallearning}.
%
%
%
%
%
%
The categorical set $\mathcal{X}^{\cat}$ is assumed to be finite and its cardinality is $\left| \mathcal{X}^{\cat} \right| = \prod_{i=1}^{n^{\cat}} \ell_i$.

Finally, a point $\x \in \mathcal{X}$ is partitioned into component by types, such that
\begin{equation}
    \x \coloneq \left(
    \x^{\cat},
    \x^{\integer}, \x^{\continuous}\right)
    \in \mathcal{X}^{\cat} \times \mathcal{X}^{\integer} \times \mathcal{X}^{\continuous}
\end{equation}
where $\mathcal{X}\coloneq\mathcal{X}^{\cat} \times \mathcal{X}^{\integer} \times \mathcal{X}^{\continuous}$ is the domain that consists of Cartesian products of the categorical set $\mathcal{X}^{\cat}$, the integer set $\mathcal{X}^{\integer}$ and the continuous set $\mathcal{X}^{\continuous}$.

Note that this work does not cover problems with \textit{meta variables} that influence the inclusion or bounds of other variables~\cite{G-2022-11, G-2024-33}.
The dimensions and bounds of the problems are fixed.

\subsection{Objectives, contributions and organization}
\label{sec:objectives_organization}

%
%
%
%
%

\mads  is a direct search algorithm that provides an exhaustive convergence theory for constrained quantitative problems~\cite{AuDe2006, AuLeDTr2018}.
Nevertheless, the direct search approach remains relatively underdeveloped for mixed-variable problems with categorical variables, especially in comparison to Bayesian optimization (BO)~\cite{GaHe2020, PeBrBaTaGu2019, RoPaDeClPeGiWy2020} and metaheuristics~\cite{DePrAgMe2002} approaches.
The mixed-variable \mads (\mvmads ), an extension of \mads  for discrete variables, has several convergence guarantees~\cite{AACW09a}.
However, \mvmads  requires the implementation of explicit \textit{user-defined neighborhoods} that define discrete variable neighborhoods.
Such neighborhoods can be tricky or arbitrary to define in a blackbox optimization context with limited information, yet they greatly impact performance.
In practice, user-defined neighborhoods are not user-friendly and may discourage practitioners from using this method.

\catmads is introduced as a framework that generalizes \mads  and improves the management of categorical variables of \mvmads .
%
The research gap addressed is the absence of a robust and rigorous optimization method for mixed-variables problems with constraints.
The main advantage of \catmads over alternative methods is its convergence guarantees and theoretical results, which BO or metaheuristic methods lack.
For instance, \catmads provides necessary optimality conditions for local minima of varying strengths.
This work introduces and adapts different notions of local optimality to develop the theoretical results of \catmads.
In doing so, several characterizations of local minima are proposed and multiple key definitions, such as hypertangent cones, are adapted to the mixed-variable setting.


\catmads addresses the research gap through five objectives.
The first objective is to automatize and generalize the user-defined neighborhood of \mvmads  with distances.
%
The second objective is to integrate the version of \mads that uses a granular mesh~\cite{AuLeDTr2018}, referred to as \gmads, to handle simultaneously integer and continuous variables efficiently and with convergence guarantees.
\mvmads  treats integers as categorical variables, thus it does not capitalize on the properties of these variables.
%
The third objective is to handle the constraints with the progressive barrier~\cite{AuDe09a}, a state-of-the-art technique with convergence guarantees. 
\mvmads uses the simpler extreme barrier strategy that simply rejects infeasible points~\cite{AuDe2006}.
The use of the progressive barrier also motivates adapting the extended poll~\cite{AuDe01a}, a mechanism exploring promising categorical variables and influencing guarantees.
%
%
%
The fourth objective is to develop a comprehensive convergence analysis that improves upon that of \mvmads  and integrates the contributions of the previous objectives.
%
In the convergence analysis, different types of local minima are developed to provide compromises between strength of local optimality and evaluation costs of \catmads.
The final objective is to propose a collection of 32 mixed-variable test problems, half unconstrained and the other half constrained.

The rest of the document is organized as follows. 
Essential elements of \mads  are discussed in \Cref{sec:intro_background} and  related work is covered in \Cref{sec:intro_related}.
The general framework \catmads is presented in \Cref{sec:cat_mads}.
%
%
%
\Cref{sec:convergence} presents the theoretical results of \catmads, along with adapted definitions from nonsmooth optimization and new types of local minima tailored for our setting.
Finally, computational experiments with state-of-the-art solvers are described in \Cref{sec:computation_exp}.
%

\subsection{Essential concepts of \mads}
\label{sec:intro_background}

%
%
%
%
\mads is a direct search method that iteratively seeks candidate points with a strict decrease of the objective function and/or constraint functions~\cite{AuDe2006}.
The candidate points lie on a mesh that discretizes the space with step sizes.
The \mads  algorithm has two main steps.
The \textit{search}, an optional step, evaluates points on the mesh using flexible strategies. 
Although optional, the search allows to incorporate efficient optimization strategies that can significantly improve the global solution quality.
Notable empirical improvements are often achieved with search steps based on quadratic model optimization~\cite{CoLed2011}.
%
%
The key step, called the \textit{poll}, evaluates candidates points around the \textit{incumbent solution}, which is the best solution found since the algorithm was launched.
This is the step that provides convergence guarantees.
If a better solution is found during the search or the poll, then it becomes the incumbent solution and the iteration is said to be successful.
Otherwise, the iteration is unsuccessful, the incumbent solution is unchanged and the mesh is refined to discretize the space more finely.
%

For the continuous case, the mesh has a single step size and the poll is constructed with a set of directions positively spanning the space.
The poll converges with a succession of unsuccessful iterations, such that candidate points are evaluated in an increasingly denser region around a minimum:~this is schematized for $\mathbb{R}^2$ 
in \Cref{fig:MADS}. 

%
\begin{figure}[htb!]
\begin{subfigure}[t]{0.32\textwidth}
\centering
\captionsetup{justification=centering}
\begin{tikzpicture}
\scalebox{1}{
    \begin{scope}
        \draw[step=2cm,gray,very thin] (0,0) grid (4,4);
        \draw[thick] (0,0) rectangle (4,4);

        \draw[dashed, -{Latex[length=3mm, width=2mm]}, thick] (2,2) -- (0,2);   
        \fill[black] (0,2) circle (2pt);       
        \draw[dashed, -{Latex[length=3mm, width=2mm]}, thick] (2,2) -- (2,0);   
        \fill[black] (2,0) circle (2pt);       
        \draw[dashed, -{Latex[length=3mm, width=2mm]}, thick] (2,2) -- (4,4);   
        \fill[black] (4,4) circle (2pt);       

        \fill[black] (2,2) circle (2pt) node[above left] {$\x_{*}$};

        \node[left] at (0.75,2.25) {$\y_{(1)}$};
        \node[below] at (2.5,0.5) {$\y_{(2)}$};
        \node[above right] at (3.25,3.25) {$\y_{(3)}$};
    \end{scope}
}
\end{tikzpicture}
%
  \caption{ }
  \label{fig:MADS_1}
\end{subfigure}
\begin{subfigure}[t]{0.32\textwidth}
\centering
\captionsetup{justification=centering}
%
\begin{tikzpicture}
\scalebox{1}{
    \begin{scope}
        \draw[step=0.5cm,gray,very thin] (0,0) grid (4,4);
        \draw[thick] (1,1) rectangle (3,3);

        \draw[dashed, -{Latex[length=3mm, width=2mm]}, thick] (2,2) -- (1,2);   
        \fill[black] (1,2) circle (2pt);                                        
        \draw[dashed, -{Latex[length=3mm, width=2mm]}, thick] (2,2) -- (2,3);   
        \fill[black] (2,3) circle (2pt);                                        
        \draw[dashed, -{Latex[length=3mm, width=2mm]}, thick] (2,2) -- (2.5,1.5); 
        \fill[black] (2.5,1.5) circle (2pt);                                     

        \fill[black] (2,2) circle (2pt) node[above left] {$\x_{*}$};

        \node[left] at (1,2.25) {$\y_{(4)}$};
        \node[below] at (2.5,1.5) {$\y_{(5)}$};
        \node[above right] at (1.5,3) {$\y_{(6)}$};
    \end{scope}
}
\end{tikzpicture}
%
  \caption{ }
  \label{fig:MADS_2} 
\end{subfigure}
%
\begin{subfigure}[t]{0.32\textwidth}
\centering
\captionsetup{justification=centering}
%
\begin{tikzpicture}
\scalebox{1}{
    \begin{scope}
        \draw[step=0.125cm,gray,very thin] (0,0) grid (4,4);
        \draw[thick] (1.5,1.5) rectangle (2.5,2.5);

        \draw[dashed, -{Latex[length=3mm, width=2mm]}, thick] (2,2) -- (1.5,2);   
        \fill[black] (1.5,2) circle (2pt);
        \draw[dashed, -{Latex[length=3mm, width=2mm]}, thick] (2,2) -- (2,1.5);   
        \fill[black] (2,1.5) circle (2pt);
        \draw[dashed, -{Latex[length=3mm, width=2mm]}, thick] (2,2) -- (2,2.5);   
        \fill[black] (2,2.5) circle (2pt);
        \draw[dashed, -{Latex[length=3mm, width=2mm]}, thick] (2,2) -- (2.5,2);   
        \fill[black] (2.5,2) circle (2pt);

        \fill[black] (2,2) circle (2pt) node[above left] {$\x_{*}$};

        \node[left] at (1.5,2) {$\y_{(7)}$};
        \node[below] at (2,1.5) {$\y_{(8)}$};
        \node[above right] at (1.5,2.5) {$\y_{(9)}$};
        \node[above right] at (2.5,1.75) {$\y_{(10)}$};
    \end{scope}
}
\end{tikzpicture}
%
  \caption{ }
  \label{fig:MADS_3}
\end{subfigure}
\caption{Three consecutive polls of \mads  converging to a local minimum $\x_{*} \in \mathbb{R}^2$.}
\label{fig:MADS}
\end{figure}
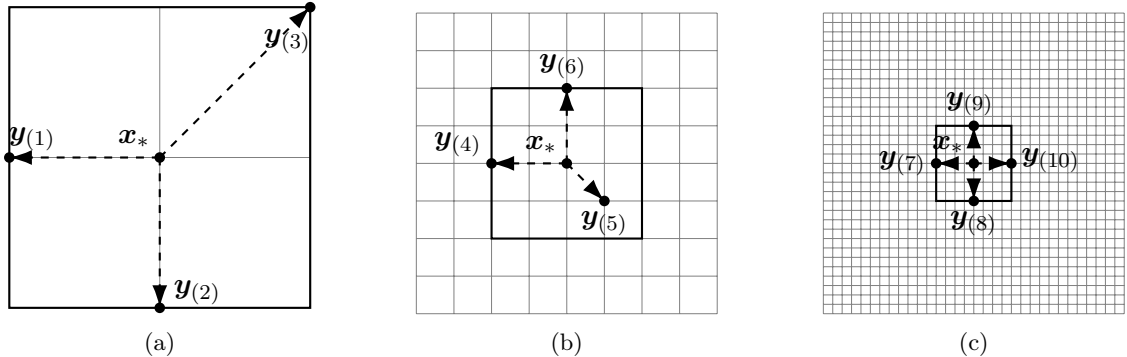
%

%
In \Cref{fig:MADS}, the poll generates candidates points with positive basis that are represented by dotted arrows starting from $\x_{\star}$. 
These candidate points $\y_{(1)}, \y_{(2)}, \ldots, \y_{(10)}$ must lie on the mesh (grey grid) and within the frame (black square). 
The frame represents the delimitation in which the poll is performed and it is greater or equal to the mesh size.
As \mads converges, the mesh and frame sizes shrink toward zero, and the step size shrinks more aggressively to force polls in denser regions: in \Cref{fig:MADS_1} there are 8 potential candidate points, and it increases to 24 in \Cref{fig:MADS_2}. 
The convergence proof is based on the last remark.
For more details on the \mads  algorithm, see Chapter 8 of~\cite{AuHa2017}.
The state-of-the art strategy to handle constraints is the progressive barrier, which is detailed in \Cref{sec:PB}.

\subsection{Related work}
\label{sec:intro_related}

The recent literature on mixed-variable blackbox problems mostly comes from BO with GPs, focusing on similarity measures, referred to as \textit{kernels}, that characterize GPs.
For continuous and integer variables, squared-exponential kernels can be used straightforwardly~\cite{GaHe2020}.
However, categorical variables require additional work.
In~\cite{GaHe2020}, categorical variables are
one-hot encoded
into binary vectors with one binary variable per category.
These binary variables are then continuously relaxed, allowing the squared-exponential kernels to be used.
Relaxed categorical variables are transformed into their original form with a transformation that, for each categorical variable, sets the highest-valued category to one, and all its other categories to zero.
%

In~\cite{ZhTaChAp2020}, each categorical variable is provided with a 2D continuous space for encoding its categories. 
The categories are mapped with a maximum likelihood procedure that positions correlated categories closely and uncorrelated categories farther apart.
%
%
The authors of~\cite{CuLeRoPeDuGl2021} propose a pre-image problem that reconstructs the encoded categorical variables.
%
This is useful in BO, since the GPs are used to form a subproblem that determines the next point to evaluate. 
If categorical variables are encoded, then the solution is also encoded.

An alternative approach to binary vectors employs matrices~\cite{PeBrBaTaGu2019, QiWuJe2008, RoPaDeClPeGiWy2020}. 
Each categorical variable is assigned a matrix whose elements are hyperparameters representing correlations or anti-correlations between its categories.
The hyperparameters are constructed with hypersphere decompositions that parametrize the matrices~\cite{QiWuJe2008}.
A parametrization with more hyperparameters offers greater modeling flexibility, but adjusting its hyperparameters becomes more costly.
See~\cite{BaDiMoLeSa2023} for in-depth presentation of different parametrizations.
This matrix approach implies that the BO subproblem is also mixed-variable.
%
%
%
In~\cite{PeBrBaTaGu2019}, the subproblem is optimized with a genetic algorithm inspired from~\cite{StNaBa1998}.

%
Genetic algorithms are also commonly employed in blackbox optimization, such as the Non-dominated Sorting Genetic Algorithm~\cite{DePrAgMe2002} (NSGA-2).
These algorithms encode points into \textit{chromosomes} that are subject to different operations inspired from the theory of evolution.
For instance, the crossover operation exchange parts of promising chromosomes to heuristically propose new solutions.  
%
%
%
%
Many other metaheuristic methods exist for the problems of interest, such as the mixed-variable Particle Swarm Optimization~\cite{WaZhZh2021}.
For a comprehensive survey, see~\cite{Ta2024}.

\section{The \catmads general framework}
\label{sec:cat_mads}

%
\catmads follows the paradigm of direct search methods:~an initialization step is done at the beginning, then search and poll steps are iteratively performed until a stopping or convergence criteria is met.
At each iteration, the algorithm seeks to improves an incumbent solution $\x_{(k)} \in \mathcal{X}$.
The framework is presented in its general form in \Cref{algo:cat_mads_algo}.
Initially, \catmads is discussed without constraints; they are addressed in \Cref{sec:PB}, after the poll step is detailed.

\begin{algorithm}[htb!]
\small

0.~\textbf{Initialization}. Initialize poll parameters, perform an initial Design of Experiment (DoE) and set $k=0$ \;

1.~\textbf{Opportunistic search} (finitely many mesh points, possibly none). \;




\vspace{0cm}

2.~\textbf{Opportunistic poll}. Perform polling over candidates points in $P_{(k)}=P_{(k)}^{\cat} \cup P_{(k)}^{\quant}$ \;



\vspace{0cm}

3.~\textbf{Opportunistic extended poll} (optional). If Steps 1. and 2. are unsuccessful, perform quantitative \;

\hspace{0.275cm} polls centered at points in $P_{(k)}^{\cat}$ with objective values sufficiently close to $f\left( \x_{(k)} \right)$ \;

\vspace{0cm}

4.~\textbf{Update}. Update poll parameters, check stopping criterion and set $k\leftarrow k+1$ \;



\hspace{0.275cm}  If the iteration is successful, increase mesh and frame parameters and \textbf{GO TO} Step 1. \;

\hspace{0.275cm} Else, decrease mesh and frame parameters, and if the mesh parameters are at their minimum sizes, then \textbf{STOP} 


\caption{The \catmads framework.}
\label{algo:cat_mads_algo}
\end{algorithm}

All the parameters of \catmads are described below.
%
A DoE is done at the initialization, allocating an initial budget of evaluations to explore the domain $\mathcal{X}$ before any optimization.
%
%
%
The DoE must contain at least one initial point $\x_{(0)}\in \mathcal{X}$, but possibly outside $\Omega$.
%

As with \mads, an iteration of \catmads is said to be successful if, at any step, a new incumbent solution $\x_{(k+1)}\in \mathcal{X}$ is found. 
In that case, \catmads proceeds directly to the update step.
This early stopping strategy is said \textit{opportunistic}, and it reduces the number of evaluations required in practice~\cite{AuHa2017}.
%
%
%
An iteration is said to be unsuccessful if no better solution is found. 
%
%
%

%
%
%
%
%

In \catmads, the search has the same purpose than \mads described in \Cref{sec:intro_background}. 
The poll $P_{(k)}$ 
is the union of the categorical poll $P_{(k)}^{\cat}$ and quantitative poll $P_{(k)}^{\quant}$.
%
The categorical poll $P_{(k)}^{\cat}$ replaces the discrete poll of \mvmads , and it uses more flexible neighborhoods based on a distance defined in \Cref{sec:distance_based_neighborhoods}.
%
Furthermore, in comparison to \mvmads , the integer variables are not grouped with the categorical variables, but rather with the continuous variables to form quantitative variables.
This modification allows to use the granular mesh of \gmads ~\cite{AuLeDTr2018} and efficiently treat integer variables.
%
The quantitative poll $P_{(k)}^{\quant}$, detailed in \Cref{sec:mads_granular_mesh}, is an adaptation of the \gmads algorithm for this work.


%
%
%
%
%
%

%
%

If both search and poll steps are unsuccessful, then the extended poll, adapted from~\cite{AuDe01a}, is deployed. 
This step performs quantitative polls around points in the categorical poll $P_{(k)}^{\cat}$ whose objective value is worse than the incumbent value $f(\x_{(k)})$ (unsuccessful), but still considered sufficiently close to $f(\x_{(k)})$.
This assessment of whether a point is sufficiently close to $f(\x_{(k)})$ is described in \Cref{sec:extended_poll}.
%
%
The test is flexible, allowing the user to control its selectivity, including the option to reject all points.
The extended poll is optional, but is labeled as a poll, since it affects the convergence guarantees.
%

\subsection{Distanced-based categorical poll}
\label{sec:distance_based_neighborhoods}

The categorical poll $P_{(k)}^{\cat}$ is characterized by two elements: (i) a categorical distance $d^{\cat} : \mathcal{X}^{\cat} \times \mathcal{X}^{\cat} \to \mathbb{R}_+$ establishing the notion of proximity between categorical components, and (ii) a parameter $m_{(k)}\in \{0,1,\ldots, \left| \mathcal{X}^{\cat} \right| -1  \}$, called the \textit{number of neighbors}, determining how many categorical components are polled.

%
%
The quality of a categorical poll depends greatly on the distance employed.  
Basic categorical distances, such as Hamming distances, only compute mismatches between categories.
\catmads does not prevent using these distances, but they may fail to provide meaningful comparisons. 
For example, if $x^{\cat} \in \{ \text{Red (R), Blue (B), Green (G)}\}$, then $d^{\cat}(\text{R}, \text{B})= 0 = d^{\cat}(\text{R}, \text{G})$.
With such distances, it is not always possible to determine which category is the closest to $\text{R}$.
In \catmads, effort needs to be spent on
constructing problem-specific categorical distances that allow useful comparisons. 
In the computational experiments in \Cref{sec:computation_exp}, the categorical distance is built by fitting a mixed-variable interpolation of the objective function~$f$.
When some knowledge of the problem is available, a user-defined distance assigning smaller values to similar categories may be appropriate.

A value of $m_{(k)}=0$ means 
the categorical poll is empty,
whereas $m_{(k)}= \left| \mathcal{X}^{\cat}\right|-1$ signifies all categorical components, except the current solution $x_{(k)}$, belong to the categorical poll.
%
Higher values allow more categorical component to be evaluated in the categorical poll,
but at greater computational costs. 
%
The number of neighbors can be based by default on the number of categorical components in the problem, \textit{e.g.}, $m_{(k)} = \left\lfloor  \left|\mathcal{X}^{\cat} \right|/2 \right\rfloor$, to scale with the problem.
This is ultimately a user choice 
that should take into account the categorical distance.
For instance, if mismatching-based distances are employed, higher values of $m_{(k)}$ may be preferable, as each neighbor is selected with limited information.
In contrast, when more sophisticated problem-specific distances are used, then each neighbor is chosen in more informed manner, which may justify using lower values of $m_{(k)}$.

The following definition introduces the structure establishing neighborhoods for categorical variables. 
%

%




\begin{definition}[Distance-induced neighborhood]
For a given number
$m \in \{0,1,\ldots, \left| \mathcal{X}^{\cat} \right| -1 \}$, 
the {\em distance-induced neighborhood} 
$\mathcal{N} \left( \x^{\cat};m \right) \subseteq \mathcal{X}^{\cat}$
is the set of 
$m+1$ 
closest categorical components to the categorical component $\x^{\cat} \in \mathcal{X}^{\cat}$ with respect to a categorical distance $d^{\cat}:\mathcal{X}^{\cat} \times \mathcal{X}^{\cat} \to \mathbb{R}_+$.

\medskip

\noindent 
%
Formally, $\left|\mathcal{N} \left( \x^{\cat};m \right)\right| = m+1$, and for any pair $\bm{u}, \bm{v} \in \mathcal{X}^{\cat}$ such that $\bm{u} \in \mathcal{N}\left(\x^{\cat};m \right)$ and $\bm{v} \not\in \mathcal{N}\left(\x^{\cat};m \right)$, the distance $d^{\cat}$ satisfies $d^{\cat} \left(\x^{\cat}, \bm{u}\right)  \leq d^{\cat} \left(\x^{\cat}, \bm{v}\right)$.

\label{def:distance_based_neighborhood}
\end{definition}

By definition, a categorical component $\x^{\cat} \in \mathcal{X}^{\cat}$ belongs to its own distance-induced neighborhood: $\x^{\cat} \in \mathcal{N}(\x^{\cat};m)$ $\text{for any } m \in \{0,1,\ldots, \left|\mathcal{X}^{\cat}\right| \}$.
The definition generalizes user-defined neighborhoods from~\cite{AACW09a}, since it suffices to construct an appropriate categorical distance inducing the user-defined neighborhoods. 
%
The categorical poll is defined from distance-induced neighborhoods.

\begin{definition}[Categorical poll]
For a given incumbent solution $\x_{(k)}=\left(\x_{(k)}^{\cat}, \x_{(k)}^{\quant} \right) \in \mathcal{X}$, the categorical poll $P_{(k)}^{\cat} \subset \mathcal{X}$ contains points whose categorical components belong to the distance-induced neighborhood $\mathcal{N}\left(\x_{(k)}^{\cat}; m_{(k)}\right)$, such that
\begin{equation*}
    P_{(k)}^{\cat} \ \coloneq \ \left( \mathcal{N}\left( \x_{(k)}^{\cat}; m_{(k)} \right) \setminus \{ \x_{(k)}^{\cat} \} \right) \times \{ \x_{(k)}^{\quant} \} \ \subset \ \mathcal{X}
\end{equation*}
where $\left| P_{(k)}^{\cat} \right| = m_{(k)}$.

\label{def:categorical_poll}
\end{definition}

%
%


In \Cref{def:categorical_poll}, the categorical poll contains points of the domain, however the incumbent quantitative component $\x_{(k)}^{\quant} \in \mathcal{X}^{\quant}$ is fixed so that categorical polls effectively modify categorical components of incumbent solutions.
%
%

\subsection{Quantitative poll and mesh}
\label{sec:mads_granular_mesh}

%
%

%
In \gmads, the mesh discretizes the quantitative space, determines the potential points 
that can be generated by the quantitative poll.
The mesh also controls the spacing between the incumbent solution and the points in these polls.
%
%
%
A step size is assigned to each variable in order to control its granularity.
For this work, the granular mesh allows to tackle simultaneously continuous and integer variables, referred to as the quantitative ones.
The following redefines the granular mesh in~\cite{AuLeDTr2018}.

\begin{definition}[Quantitative mesh~\cite{AuLeDTr2018}]
For a given iteration $k \geq 0$, the {\em quantitative mesh}, centered at the quantitative component $\x_{(k)}^{\quant} \in \mathcal{X}^{\quant}$, discretizes the quantitative set $\mathcal{X}^{\quant}$, as follows
\begin{equation}
    M_{(k)}^{\quant} \, \coloneq \left\{ \x_{(k)}^{\quant} +  \diag\left( \bm{\delta}_{(k)} \right) \z \, : \, \z \in \mathbb{Z}^{n^{\quant}}  \right\} \ \subset \, \mathcal{X}^{\quant}
    \label{eq:quantitative_mesh_compact}
\end{equation}
where $\bm{\delta}_{(k)}\coloneq \left( \delta_{(k),1}, \delta_{(k), 2}, \ldots, \delta_{(k),n^{\quant}} \right) \in \mathbb{R}_+^{n^{\quant}}$ is the mesh size vector and $\delta_{(k),i} \in \mathbb{R}_+$ is the mesh size parameter of the $i$-th quantitative variable $x_{(k),i}^{\quant} \in \mathcal{X}_i^{\quant}$.

\label{def:quantitative_mesh}
\end{definition}

In \Cref{def:quantitative_mesh}, the step size of an integer variable belongs more precisely to set of integers $\mathbb{Z}$ and its minimal value is one.
%
%
%
For clarity, \Cref{eq:quantitative_mesh_compact} from \Cref{def:quantitative_mesh} may be expanded as follows to distinguish integer and continuous variables
\begin{equation*}
    M_{(k)}^{\quant} \coloneq \left\{ \y^{\quant}=\left( \x_{(k)}^{\integer} +  \diag\left( \bm{\delta}_{(k)}^{\integer} \right) \z^{\integer}, \ \x_{(k)}^{\continuous} +  \diag\left( \bm{\delta}_{(k)}^{\continuous} \right) \z^{\continuous} \right) \, : \, \z^{\integer} \in \mathbb{Z}^{n^{\integer}}, \ \z^{\continuous} \in \mathbb{Z}^{n^{\continuous}}  \right\}
\end{equation*}
where $\bm{\delta}_{(k)}^{\integer} \in \mathbb{N}^{n^{\integer}}$ and $\bm{\delta}_{(k)}^{\continuous} \in \mathbb{R}_+^{n^{\continuous}}$ are respectively the integer and continuous mesh size vectors.

In \Cref{def:quantitative_mesh}, the quantitative mesh $M_{(k)}^{\quant}$ is a subset of the quantitative set $\mathcal{X}^{\quant}$, as it represents a discretization of the quantitative variables.
Similarly to the categorical poll, the quantitative poll contains point in the domain. 
Hence, to formulate the quantitative poll on points, this next definition extends \Cref{def:quantitative_mesh} to discretize the domain~$\mathcal{X}$.
%
\begin{definition}[Mesh~\cite{AACW09a}]
For a given iteration $k \geq 0$, the {\em mesh} discretizes the domain $\mathcal{X}$ is formed as Cartesian product of the categorical set and the quantitative mesh, such that $M_{(k)} \coloneq \mathcal{X}^{\cat} \times M_{(k)}^{\quant} \subset \mathcal{X}.$
\label{def:mesh}
\end{definition}

Next, the quantitative poll is defined from the mesh.
This poll uses the standard mechanisms of \mads . 
Mathematically, candidate points, which lie on the mesh, are constructed by translating a quantitative incumbent solution $\x_{(k)}^{\quant}$ with directions from a positive basis.

\begin{definition}[Quantitative poll~\cite{AACW09a}]
For a given incumbent solution $\x_{(k)}=\left(\x_{(k)}^{\cat}, \x_{(k)}^{\quant} \right) \in \mathcal{X}$, the {\em quantitative poll} $P_{(k)}^{\quant} \subset \mathcal{X}$ contains points whose quantitative components are generated from a positive spanning set of directions $\mathcal{D}_{(k)}$, such that
\begin{equation*}
    %
    P_{(k)}^{\quant} \coloneq \left\{ \left(\x^{\cat}, \x_{(k)}^{\quant} + \diag\left( \bm{\delta}_{(k)} \right) \bm{d} \right) \in M_{(k)}   \, : \, \bm{d} \in \mathcal{D}_{(k)}  \right\} \subset \mathcal{X}.
\end{equation*}
%

\noindent Each direction $\bm{d} \in \mathcal{D}_{(k)}$, scaled by the mesh size vector $\bm{\delta}_{(k)}$, must be within the geometry of the frame, characterized by frame size vector
$\bm{\Delta}_{(k)} \geq \bm{\delta}_{(k)}$, such that
\begin{equation}
-\bm{\Delta}_{(k)} \leq \diag\left( \bm{\delta}_{(k)} \right) \bm{d} \leq \bm{\Delta}_{(k)},
\label{eq:direction_inequality}
\end{equation}
%
%
%
%
%
%
%
where $\bm{\Delta}_{(k)}\coloneq \left( \Delta_{(k),1}, \Delta_{(k), 2}, \ldots, \Delta_{(k),n^{\quant}} \right)~\in~\mathbb{R}_{+}^{n^{\quant}}$ and $\Delta_{(k),i}$ is the frame size parameter of the $i$-th quantitative variable $x_{(k),i}^{\quant} \in \mathcal{X}_i^{\quant}$ for $i \in I^{\quant}$.

\label{def:quantitative_poll}
\end{definition}
%

%
%
%
%
%
%
%
%
For the quantitative variables, the management of parameters is done variable-wise in order to consider the variable type and the scale of the variables.
%
The frame sizes of an integer and continuous variables are respectively restricted to the following discrete sets $\Delta_{(k),i}^{\integer} \in \{ a \times 10^b~:~a \in \{1,2,5\}, \ b\in \mathbb{N} \}$, with $\Delta_{(k),i}^{\integer} \geq 1$,  for $i \in I^{\integer}$,
and
$\Delta_{(k),j}^{\continuous} \in \{ a \times 10^b~:~a \in \{1,2,5\}, \ b\in \mathbb{Z} \}$ for $j \in I^{\continuous}$.
%
%
%
%
%
%
%
%
%
%
A step size is always lower than or equal to its corresponding frame size, \textit{i.e.}, $\delta_{(k),i}^t \leq \Delta_{(k),i}^t$.
%
The frame and step sizes parameters
with specific rules detailed in~\cite{AuLeDTr2018}.
%
%
The frame and step sizes are decreased on unsuccessful iterations, and increased on successful ones.
%
%
%
%
A step size must decrease toward zero more rapidly than its frame size to ensure convergence: typically,
it decreases toward zero with an order of the square of its frame size~\cite{AuHa2017, AuLeDTr2018}.
%
%
%
%
The initialization of the frame and step sizes considers the bounds and the initial values of $\x_{(0)}$ or a DoE:
%
%
%
%
%
%
%
see~\cite{AuLedTr2014, AuLeDTr2018} for more details.

%
%
%
By default, \gmads  uses Householder transformations to construct positive spanning set of directions $\mathcal{D}_{(k)}$.
%
At each iteration $k$, a random and normalized vector $\bm{v}_{(k)} \in \mathbb{R}^{n^{\quant}}$ is generated.
From $\bm{v}_{(k)}$, the Householder transformation $I - 2\bm{v}_{(k)}\bm{v}_{(k)}^{\top}$ is done.
Afterwards, each column of the matrix is used to create two opposing directions.
These directions are rescaled with the frame and step sizes parameters to ensure that the inequality in~\eqref{eq:direction_inequality} is respected.
%
This ensures that the poll points belong to the mesh and reside within the frame, as in Figure~\ref{fig:MADS}.

%

%
%

%
%

\subsection{Domination and constraint handling with the progressive barrier}
\label{sec:PB}

%
The strategy to handle the constraints, called the progressive barrier (PB)~\cite{AuDe09a}, is now described.
The PB works with two incumbent solutions, instead of a single one $\x_{(k)} \in \mathcal{X}$ as previously: the feasible incumbent solution $\smash{\x_{\feasible} \in \Omega}$ and the infeasible incumbent solution $\smash{\x_{\infeasible} \in \mathcal{X} \setminus \Omega}$.
For readability, the subscript $k$ is discarded when superscripts $\feasible$ and $\infeasible$ are used.
The key idea of the PB is to perform two polls.
%
The \textit{feasible poll} is centered at
$\smash{\x_{\feasible}}$ and is done as described in previous sections.
The \textit{infeasible poll} is centered at
$\smash{\x_{\infeasible}}$, and this solution is updated to progressively approach feasibility throughout the iterations.

In Problem~\eqref{eq:formulation}, recall that a point $\x \in \mathcal{X}$ is feasible if and only if all the constraints are respected.
%
Hence, a point $\y \in \mathcal{X}$ is infeasible if at least one constraint is not respected, \textit{i.e.}, there exists $j \in J$ with $g_j(\x)>0$.
The infeasibility of a point does not necessarily imply that it is not promising.
In fact, optimal solutions often lie on the boundary of the feasible set.
%
In order to quantify the infeasibility of a point, the constraint violation function $h:\mathcal{X} \to \overline{\mathbb{R}}$ is introduced as follows
\begin{eqnarray}
h(\x) \coloneq 
\begin{cases}
    \begin{array}{ll}
       \sum\limits_{j \in J} \left( \max \{0, g_j(\x) \} \right)^2 &  \text{ if } \x \in \mathcal{X}, \\
       \infty   & \text{ otherwise }
    \end{array}
\end{cases} 
\label{eq:fonction_agg}
\end{eqnarray}
with $h(\x) = 0 \Leftrightarrow \x \in \Omega$ (feasible) and $h(\x)>0 \Leftrightarrow \x \in \mathcal{X} \setminus \Omega$ (infeasible). 
%
A point $\x \not \in \mathcal{X}$ is valued as $h(\x)=\infty$, which conveniently outlines that it is not a solution considered.

The PB introduces a threshold parameter $h_{(k)}^{\text{max}}$, called the \textit{barrier}.
This parameter controls the acceptable infeasibility of $\smash{\x_{\infeasible}}$.
More precisely, the feasible and infeasible solutions are determined with the following expressions:
\begin{equation}
    \x_{\feasible} \in \argmin \left\{ f(\x) \, : \, \x\in \Omega \cap \mathbb{X}  \right\} \quad
    \text{and} \quad
    \x_{\infeasible} \in \argmin \left\{ f(\x) \, : \, 0 < h(\x) \leq h_{(k)}^{\text{max}}, \, \x \in \mathbb{X}   \right\}
\label{eq:best_feasible_and_infeasible}
\end{equation}
where $\mathbb{X} \subset \mathcal{X}$, called the \textit{set of known points}, is the set of points whose objective and constraint functions were evaluated since the algorithm was launched.
%
%
The barrier is initialized at $h_{(0)}^{\text{max}}=\infty$ and 
it is reduced iteratively toward zero.

The update of the barrier and the mesh parameters is done with notions of dominance~\cite{AuHa2017}.
%
%
\begin{definition}[Feasible and infeasible dominance]
The point $\x \in \Omega$ is said to {\em dominate} $\y \in \Omega$, denoted $\x \prec_f \y$, if $f(\x) < f(\y)$.
The point $\x \in \mathcal{X} \setminus \Omega$ is said to dominate $\y \in \mathcal{X} \setminus \Omega$, denoted $\x \prec_h \y$, if $f(\x) \leq f(\y)$ and $h(\x) \leq h(\y)$ with a least one strict inequality.

\label{def:dominations}
\end{definition}
In the previous sections without constraints, an iteration is either successful or unsuccessful.
In the constrained case, a successful iteration is divided in two cases based on \Cref{def:dominations}:
\begin{enumerate}
    \item an iteration is \textit{dominating} if either a candidate point $\y \in \Omega$ with $\y \prec_f \x_{\feasible}$ or a candidate point $\y \in \mathcal{X} \setminus \Omega$ with $\smash{\y \prec_{h} \x_{\infeasible}}$ is found at any step.
    The barrier is updated as $h_{(k+1)}^{\text{max}} \coloneq h(\x_{\infeasible})$, and the mesh and frame sizes parameters are increased.

    \item an iteration is \textit{improving} if it is not dominating and if a candidate point $\y \in \mathcal{X} \setminus \Omega$ that improves the constraint violation function, such that $0<h(\y)<h(\x_{\infeasible})$ is found at the end of an iteration. 
    In this case, the mesh and frame parameters are unchanged, and the barrier is updated after going through all the steps with $h_{(k+1)}\coloneq \max \{ h(\bm{v})~:~h(\bm{v})<h(\x_{\infeasible}),~\bm{v}\in \mathbb{X}\}$.

\end{enumerate}
The opportunistic strategy is reused with the PB, however it can only be applied to dominating iterations.
An iteration is unsuccessful if it is not dominating or improving.
In that case, the parameters of \gmads  are decreased and the barrier is unchanged.

%

%
With the PB, the poll set $P_{(k)}$ from \Cref{algo:cat_mads_algo} becomes
$P_{\feasible} \cup P_{\infeasible}$,
where the feasible poll set $P_{\feasible}$ and infeasible poll set $P_{\infeasible}$ are respectively centered at the feasible and infeasible solutions.
The infeasible and feasible polls both have their own quantitative and categorical polls.
The feasible and infeasible categorical polls have their own number of neighbors.
Similarly, the feasible and infeasible quantitative polls have their own set of directions.

The integration of the PB within \catmads is achieved by 1)~considering improving iterations, 2)~performing two polls instead of one and 3)~generalizing the extended poll from~\cite{AuDe01a} to use the PB instead of the extreme barrier, as detailed in the next section.

%
%
%

\subsection{Extended poll}
\label{sec:extended_poll}

The extended poll
is optionally performed following unsuccessful search and poll steps.
\Cref{subfig:extended_poll_example} illustrates the extended poll on the \text{RG example} in which the red and green planes represent two quantitative variables in $\mathbb{R}^2$ corresponding to the two colors (categories).
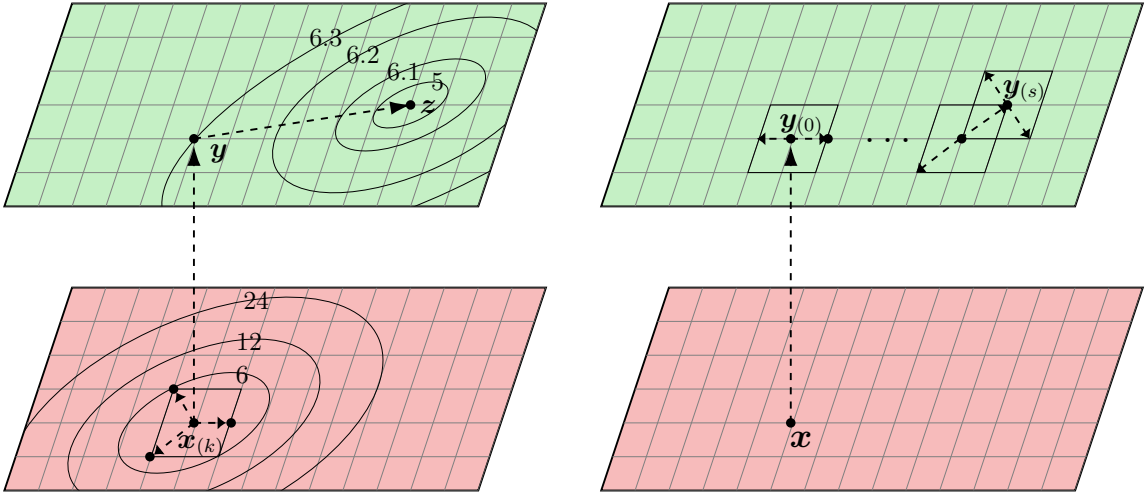
\begin{figure}[htb!]
\captionsetup[subfigure]{skip=0pt}
\captionsetup{skip=0pt} 
\begin{subfigure}[t]{0.45\textwidth}
\centering
  \scalebox{0.9}{\begin{tikzpicture}

    \definecolor{myblue}{RGB}{51,171,255}   
    \definecolor{myred}{RGB}{236,87,87}     
    \definecolor{mygreen}{RGB}{112,219,112} 

    \def\PlaneShape{(-3,-1) -- (4,-1) -- (5,2) -- (-2,2) -- cycle;}

    \newcommand{\DrawGrid}{%
        \foreach \x in {-3,-2.5,-2,-1.5,-1,-0.5,0,0.5,1,1.5,2,2.5,3,3.5} {
            \draw[black!50] ({\x+0.5},-1) -- ({\x+1.5},2);
        }
        \foreach \y in {-1,-0.5,0,0.5,1,1.5,2} {
            \draw[black!50] (-3,\y) -- (5,\y);
        }
    }

    \fill[myred,opacity=0.4] \PlaneShape;
    \draw[black, thick] \PlaneShape;
    \begin{scope}
        \clip \PlaneShape;
        \DrawGrid
    \end{scope}

    \begin{scope}[shift={(0,4.2)}]
        \fill[mygreen,opacity=0.4] \PlaneShape;
        \draw[black, thick] \PlaneShape;
        \begin{scope}
            \clip \PlaneShape;
            \DrawGrid
        \end{scope}
    \end{scope}
    
    \begin{scope}
        \clip \PlaneShape;

        \begin{scope}[rotate around={25:(-0.2,0)}]
        \draw[black] (-0.2,0) ellipse (1.2 and 0.6);
        \node at (0.75,0.35) {6};
        \draw[black] (-0.2,0) ellipse (2 and 1);
        \node at (1.05,0.75) {12};
        \draw[black] (-0.2,0) ellipse (3 and 1.5);
        \node at (1.4, 1.25) {24};
        %
        \end{scope}
        
    \end{scope}

    \begin{scope}[shift={(0,4.2)}]
        \clip \PlaneShape;
        \begin{scope}[rotate around={25:(3,4.8-4.3)}]
            \draw[black] (3,4.8-4.3) ellipse (0.6 and 0.25); 
            \draw[black] (3,4.8-4.3) ellipse (1.2 and 0.5);
            \draw[black] (3,4.8-4.3) ellipse (2.2 and 1);
            \draw[black] (3,4.8-4.3) ellipse (4 and 1.4);
        \end{scope}
        \node at (1.75, 1.5) {6.3};
        \node at (2.3, 1.25) {6.2};
        \node at (2.9, 1) {6.1};
        \node at (3.4, 0.85) {5};
    \end{scope}
    
    \fill[black] (-0.2,0) circle (2pt);
    \node[below right] at (-0.575, -0.05) {\large $\x_{(k)}$};
    
    \fill[black] (-0.2,4.2) circle (2pt);
    \node[below right] at (-0.1,4.25) {\large $\y$};
    
    \fill[black] (3,4.7) circle (2pt);
    \node[below right] at (3,4.9) {\large $\z$};
    
    \draw[dashed,-{Latex[length=3mm, width=2mm]},thick] (-0.2,0) -- (-0.2,4.1);
    \draw[dashed,-{Latex[length=3mm, width=2mm]},thick] (-0.2,4.2) -- (3,4.7);

     \def\parallelogram(#1,#2){%
        \draw (#1-0.8333,#2-0.5) -- 
              (#1+0.1666,#2-0.5) -- 
              (#1+0.5,#2+0.5) -- 
              (#1-0.5,#2+0.5) -- 
              cycle;
    }

    \parallelogram(0,0)
    \draw[dashed,-{Latex[length=1.25mm, width=2mm]},thick] (-0.2,0) -- (0.35-0.075,0);
    \fill[black] (0.35,0) circle (2pt);
    
    \draw[dashed,-{Latex[length=1.25mm, width=2mm]},thick] (-0.2,0) -- (-0.5+0.04,0.5-0.05);
    \fill[black] (-0.5,0.5) circle (2pt);
    
    \draw[dashed,-{Latex[length=1.25mm, width=2mm]},thick] (-0.2,0) -- (-0.85+0.06, -0.5+0.03);
    \fill[black] (-0.85,-0.5) circle (2pt);

\end{tikzpicture}}
  \medskip
  \subcaption{Extended poll initiated at $\y \in P_{(k)}^{\cat}$, successfully terminating at $\z$ where $4=f(\z) < f(\x_{(k)})=5$.}
  \label{subfig:extended_poll_example}
\end{subfigure}
\hspace{0.25cm}
\begin{subfigure}[t]{0.45\textwidth}
\centering
    \scalebox{0.9}{\begin{tikzpicture}
    \definecolor{myblue}{RGB}{51,171,255}   
    \definecolor{myred}{RGB}{236,87,87}     
    \definecolor{mygreen}{RGB}{112,219,112} 

    \def\PlaneShape{(-3,-1) -- (4,-1) -- (5,2) -- (-2,2) -- cycle;}

    \newcommand{\DrawGrid}{%
        \foreach \x in {-3,-2.5,-2,-1.5,-1,-0.5,0,0.5,1,1.5,2,2.5,3,3.5} {
            \draw[black!50] ({\x+0.5},-1) -- ({\x+1.5},2);
        }
        \foreach \y in {-1,-0.5,0,0.5,1,1.5,2} {
            \draw[black!50] (-3,\y) -- (5,\y);
        }
    }

    \fill[myred,opacity=0.4] \PlaneShape;
    \draw[black, thick] \PlaneShape;
    \begin{scope}
        \clip \PlaneShape;
        \DrawGrid
    \end{scope}
    
    \begin{scope}[shift={(0,4.2)}]
        \fill[mygreen,opacity=0.4] \PlaneShape;
        \draw[black, thick] \PlaneShape;
        \begin{scope}
            \clip \PlaneShape;
            \DrawGrid
        \end{scope}
    \end{scope}

    \fill[black] (-0.2,0) circle (2pt);
    \node[below right] at (-0.35,0) {\Large $\x$};
    
    \fill[black] (-0.2,4.2) circle (2pt);
    \node[below right] at (-0.5,4.7) {\large $\y_{(0)}$};

    \fill[black] (-0.2+0.55,4.2) circle (2pt);
    %

    \node[below right] at (0.75, 4.35) {\LARGE $\ldots$};

     \fill[black] (-0.2+2.525,4.2) circle (2pt);
    %

    \node[below right] at (-0.2+3, 5.2) {\large $\y_{(s)}$};
     \fill[black] (3,4.7) circle (2pt);

    \draw[dashed,-{Latex[length=3mm, width=2mm]},thick] (-0.2,0) -- (-0.2,4.1);

    \def\parallelogram(#1,#2){%
        \draw (#1-0.8333,#2-0.5) -- 
              (#1+0.1666,#2-0.5) -- 
              (#1+0.5,#2+0.5) -- 
              (#1-0.5,#2+0.5) -- 
              cycle;
    }
    \parallelogram(0,4.2)
    \parallelogram(2.5,4.2)
    \parallelogram(3.16666,4.7)

    \draw[dashed,-{Latex[length=1.25mm, width=2mm]},thick] (-0.2,4.2) -- (0.35,4.2);
     \draw[dashed,-{Latex[length=1.25mm, width=2mm]},thick] (-0.2,4.2) -- (-0.7,4.2);

    \draw[dashed,-{Latex[length=1.25mm, width=2mm]},thick] (-0.2+2.5,4.2) -- (0.35+2.66,4.2+0.5);
     \draw[dashed,-{Latex[length=1.25mm, width=2mm]},thick] (-0.2+2.5,4.2) -- (-0.7+2.35,4.2-0.5);

     \draw[dashed,-{Latex[length=1.25mm, width=2mm]},thick] (3,4.7) -- (3-0.33,4.7+0.5);
     \draw[dashed,-{Latex[length=1.25mm, width=2mm]},thick] (3,4.7) -- (3+0.33,4.7-0.5);
    
\end{tikzpicture}}
    \medskip
  \subcaption{The sequence of extended poll points initiated at $y = \y_{(0)}$, and terminating at $\z= \y_{(s)}$.}
  \label{subfig:extended_poll_quantpolls}
\end{subfigure}
\caption{The RG example for the extended poll with the mesh in gray.}
\label{fig:extended_poll_all}
\end{figure}
In the figure, the incumbent solution $\x_{(k)}$ is in the red plane.
Both the quantitative poll and categorical poll at $\y \in P_{(k)}^{\cat}$ in the green plane fail to improve $\x_{(k)}$. 
The extended poll initiated at  $\y \in P_{(k)}^{\cat}$ generates a series of quantitative polls in the green plane, successfully terminating at $\z$ with $4=f(\z) < f(\x_{(k)})=5$.

An important drawback of the extended poll is its potential high costs, especially when dealing with many variables.
To mitigate these costs, points from the categorical poll that undergo quantitative polls are selected based on a test.
%
%
The test accepts points 
whose objective value is worse, but sufficiently close to an incumbent value.

\begin{definition}[$\xi$-extended poll trigger]
For a given parameter $\xi \in \overline{\mathbb{R}}$, 
the {\em $\xi$-extended poll trigger function} 
$t_{\xi}:\mathcal{X} \times \mathcal{X} \to \{0,1\}$ 
at the unsuccessful poll point $\y \in P_{(k)}^{\cat} = P_{\feasible}^{\cat} \cup P_{\infeasible}^{\cat}$ is defined as
\begin{equation}
    t_{\xi} \left( \x_{(k)}, \y\right) \coloneq 
        \left\{
        \begin{array}{cl}
            1  &\mbox{ if } \ 0 \leq f\left( \y \right) - f \left( \x_{(k)} \right) \leq \xi \left| f \left( \x_{(k)} \right) \right|,\\
            0 &\mbox{ otherwise. } \\  
        \end{array}
        \right.
\label{eq:extended_poll_test}
\end{equation}
%
\label{def:extended_poll_test}
\end{definition}

%
%
The parameter $\xi \in \overline{\mathbb{R}}$ has the following impact:
\begin{itemize}
    \item if $\xi < 0$, then no point passes the test, since the two inequalities are mutually exclusive;
    

    \item if $\xi = \infty$, then all points trivially pass the test;
    
    %

    \item if $\xi \in \mathbb{R}_+$, the test checks if the relative deterioration is within a ratio of $\xi$. 
    
\end{itemize}
In \catmads, the extended poll can be deactivated by setting $\xi < 0$.
In \mvmads, it is always activated and the default implementation proposed is 
$\xi=0.05$, meaning that a deterioration of 5\% on the best feasible value is accepted.

In this work, the extended poll is generalized with the PB containing two incumbent solutions.
To take into account the PB, the $\xi$-extended poll triggers are performed by considering wether the unsuccessful points are feasible or not. 
If an unsuccessful point $\y \in P_{(k)}^{\cat}$ is feasible, then the $\xi$-extended poll trigger is performed with the feasible incumbent $\x_{\feasible}$.
Otherwise, it is performed with the infeasible solution $\x_{\infeasible}$, and in that case, the point must also respect the barrier to be selected.
Mathematically, the set of selected points undergoing extended polls is expressed as
%
    %
%
\begin{align}
\begin{split}
    \mathcal{E}_{(k)} &\coloneq \left\{ \y \in P_{(k)}^{\cat} \, : \,  \left( t_{\xi} \left( \x_{\feasible}, \y \right)=1 \text{ and } h\left( \y \right) =0 \right) \, \text{or} \,
    \left( t_{\xi} \left( \x_{\infeasible}, \y \right)=1 \text{ and } 0<h\left( \y \right) \leq h_{(k)} \right) \right\}.
\label{eq:selected_points}
\end{split}
\end{align}
where $h(\y)=0$ signifies that $\y$ is feasible.

%
%
%
%
%
    %
    %

Once selected for the extended poll, a point $\y \in \mathcal{E}_{(k)}$ undergoes a finite sequence of quantitative polls initiated around it.
A quantitative poll is performed until a new incumbent solution is opportunistically found, or until it does not produce a dominating point.
This is formalized in the next definition.
%
%
%
\begin{definition}[Extended poll initiated at $\y$~\cite{AuDe01a}]
For a given iteration $k>0$,
the {\em extended poll initiated at $\y \coloneq \y_{(0)} \in \mathcal{E}_{(k)}$} produces a finite sequence of iterates 
$\y_{(0)}, \y_{(1)}, \y_{(2)}, \ldots, \y_{(s)}$, such that 
$\y_{(j)} \prec \y_{(j-1)} \, 
\text{for all } j \in S \coloneq \{1,2,\ldots,s\}$ and $\x_{(k)} \prec \y$.
The operator $\prec$ and the incumbent solution $\x_{(k)}$ are either:
\begin{itemize}
    \item the feasible dominance $\prec_{f}$ and the feasible incumbent solution $\x_{\feasible} \in \Omega$, if $\y \in \Omega$;

    \item or the infeasible dominance $\prec_{h}$ and the infeasible incumbent solution $\x_{\infeasible} \in \mathcal{X} \setminus \Omega$, if $\y \in \mathcal{X} \setminus \Omega$. 
\end{itemize}
\medskip
The point $\z \coloneq \y_{(s)} \in \mathcal{X}$, called an {\em endpoint}, is either a new incumbent solution or the quantitative poll at $\z$
fails to produce a dominating point, such that
$$\bm{w} \not \prec \z,   \, \text{for all } \bm{w} \in \left\{ \left(\y^{\cat}, \z^{\quant} + \text{diag} \left( \bm{\delta}_{(k)} \right) \bm{d} \right) \, : \, \bm{d} \in \mathcal{D} \left( \z \right) \right\}$$
where $\mathcal{D} \left( \z \right)$ is the set of directions used for $\z$.

\label{def:extended_poll_point}
\end{definition}

From \Cref{def:extended_poll_point}, it follows that, for a given iteration $k>0$, an extended poll initiated at $\y \in \mathcal{E}_{(k)}$  generates an union of quantitative polls performed at $\y_{(0)}, \y_{(1)}, \ldots, \y_{(s)}$, denoted $\mathcal{E}_{(k)}(\y)$ and expressed as 
{\small
\begin{align*}
    \mathcal{E}_{(k)} \left( \y \right) \coloneq \bigcup_{j \in S} \Big\{ \left( \y^{\cat}, \y_{(j)}^{\quant} + \diag \left( \bm{\delta}_{(k)} \right) \bm{d} \right) \in M_{(k)}  : \
    &\bm{d} \in \mathcal{D} \left( \y_{(j)} \right), \\[-0.25cm]
    &\y_{(j)}^{\quant} = \y_{(j-1)}^{\quant} + \diag \left( \bm{\delta}_{(k)} \right) \bar{\bm{d}} \text{ for some } \bar{\bm{d}} \in  \mathcal{D} \left( \y_{ \left(j-1 \right)} \right), \\
    &\y_{(j)} \prec \y_{ \left(j-1 \right)}  \Big\} \subset \mathcal{X}.
\end{align*}
}

In \Cref{subfig:extended_poll_quantpolls}, the RG example is reused to illustrate \Cref{def:extended_poll_point}.
The sequence of polls is initiated at $\y \in P_{(k)}^{\cat}$ and yields improvements until the quantitative polls produces a new incumbent solution $\z$.
The quantitative polls have their own sets of two directions.
In comparison to quantitative polls in \Cref{def:quantitative_poll}, extended polls do not require a positively spanning set of directions.
This allows using fewer directions, thereby reducing their cost.
In the example, the mesh is not increased with improvements of the extended poll. 
In theory, it could be increased, as long as the extended poll ends with the initial mesh and frame sizes.

From \Cref{def:extended_poll_point} and the sets of selected points in~\eqref{eq:selected_points}, the extended poll $\bar{P}_{(k)}$ at iteration $k$ is defined as the union of extended polls initiated at the selected points, such that
    %
%
\begin{equation}
    \bar{P}_{(k)} \coloneq  \bigcup_{\y \in \mathcal{E}_{(k)}} \mathcal{E}_{(k)} \left( \y  \right)
\label{eq:extended_poll}
\end{equation}
where $\mathcal{E}_{(k)} \left( \bm{y}  \right)$ in \cref{def:extended_poll_point} adjusts its operator $\prec$ to either $\prec_{f}$ or $\prec_{h}$, depending on whether $\y$ is feasible or not.

The extended poll can be viewed as a mixed-variable poll since it proposes points that modify both the categorical and quantitative variables of incumbent solutions.
In \catmads, these polls offer additional optimization mechanisms that balance evaluation costs and local solution quality.
This becomes apparent in the next section, as they improve convergence guarantees at some additional cost.

\section{Convergence analysis}
\label{sec:convergence}

%

The convergence analysis of \catmads follows the structure used in~\cite{AACW09a,AuDe2006}.
It relies on standard assumptions: an initial point $\x_{(0)} \in \mathcal{X}$ with $f\left(\x_{(0)}\right) \in \mathbb{R}$ is available, and all iterates $\{ \x_{(k)} \}$ lie in a compact set.
The convergence analysis is divided into four parts.
First, different types of local minima of varying strengths are presented.
Second,  \Cref{sec:analysis_refining_subsequence} redefines and shows the existence of a refining subsequence.
Third, convergence results with respect to the continuous variables are derived in \Cref{sec:analysis_calculus}, using notions of nonsmooth calculus.
Finally, \Cref{sec:analysis_interactions} strengthens the convergence results by studying the extended poll.

\subsection{Types of mixed-variable local minima}
\label{sec:local_minima}

%

In the following, an open ball for the continuous variables is denoted $\mathcal{B} \left( \bm{u} ; \epsilon \right) \coloneq \{ \bm{v} \in \mathbb{R}^{n^{\continuous}}  : \, \| \bm{u} - \bm{v} \|_2 < \epsilon \}$, and a neighborhood for integer variables is defined as $\mathcal{I} \left( \bm{u} ; \eta \right) \coloneq \{ \bm{v} \in \mathbb{Z}^{n^{\integer}}  : \, \| \bm{u} - \bm{v} \|_1 \leq \eta\}$.
In the mixed-variable setting, a local minimum is defined as follows.

\begin{definition}[Local minimum]
A point $\x=\left(\x^{\cat}, \x^{\integer}, \x^{\continuous} \right) \in \Omega$
is a {\em mixed-variable local minimum} if there exists a triplet $m \in \mathbb{N}^{*}$, $\eta \in \mathbb{N}^*$, $\epsilon \in \mathbb{R}_+^*$, such that
%
\begin{eqnarray*} 
\begin{array}{ll}    
   f(\x) \leq f(\y), &\text{for all } \y  \hspace{0.1cm} \in \left( \mathcal{N}\left(\x^{\cat}; m\right) \times \mathcal{I} \left( \x^{\integer};\eta \right) \times \mathcal{B}\left(\x^{\continuous}; \epsilon\right) \right) \, \cap \,  \Omega.   \\
   
\end{array}
\end{eqnarray*}

\label{def:complete_local_minimum}
\end{definition}

In \Cref{fig:min_complete}, a local minimum is  schematized for the \textit{RBG example} containing a single categorical variable $x^{\cat} \in \{\text{Red}, \text{Green}, \text{Blue} \}$, a single integer variable $x^{\integer} \in \{0,1,2\}$ and no relaxable constraint, \textit{i.e.}, $\Omega= \mathcal{X}$.
The planes represent the set $\mathcal{X}^{\continuous} \subset \mathbb{R}^{2}$ for fixed categorical and integer variables.
The gray zones depicts continuous balls (unconstrained)
and dotted arrows represent the couples of categorical and integer components involve in the minimum.

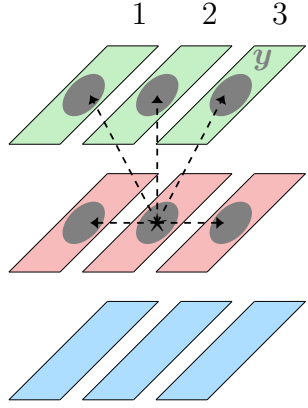
\begin{figure}[htb!]
\centering
\scalebox{0.85}{\begin{tikzpicture}

\clip (-1.55, -0.1) rectangle (3.1, 7);

\newcommand{\FACE}[5]{ 
  \node[] (C) at (#1, #2, #3) {}; 
  \fill[#5,opacity=0.4] 
    ($(C) + (#4*0.20,#4,#4)$) -- 
    ($(C) + (0,#4,#4)$) -- 
    ($(C) + (0,#4,0)$) -- 
    ($(C) + (#4*0.20,#4,0)$) -- 
    cycle;
  \draw[-] ($(C) + (#4*0.20,#4,#4)$) -- ($(C) + (0,#4,#4)$);
  \draw[-] ($(C) + (#4*0.20,#4,#4)$) -- ($(C) + (#4*0.20,#4,0)$);
  \draw[-] ($(C) + (0,#4,0)$) -- ($(C) + (0,#4,#4)$);
  \draw[-] ($(C) + (0,#4,0)$) -- ($(C) + (#4*0.20,#4,0)$);
}

\newcommand{\FaceSize}{4}

\definecolor{blueface}{RGB}{51,171,255}
\definecolor{redface}{RGB}{236,87,87}
\definecolor{greenface}{RGB}{112,219,112}

\def\xleft{0}
\def\xmid{1.15}
\def\xright{2.3}

\FACE{\xleft}{2}{0}{\FaceSize}{greenface}
\FACE{\xleft}{0}{0}{\FaceSize}{redface}
\FACE{\xleft}{-2}{0}{\FaceSize}{blueface}
\filldraw[gray, thick, rotate around={-45:(\xleft+\FaceSize*0.10, 0+\FaceSize, \FaceSize/2)}]
  (\xleft+\FaceSize*0.10, 0+\FaceSize, \FaceSize/2) ellipse (6.5pt and 11pt);
\filldraw[gray, thick, rotate around={-45:(\xleft+\FaceSize*0.10, 2+\FaceSize, \FaceSize/2)}] 
  (\xleft+\FaceSize*0.10, 2+\FaceSize, \FaceSize/2) ellipse (6.5pt and 11pt);

\FACE{\xmid}{2}{0}{\FaceSize}{greenface}
\fill[black] ($(\xmid,2,0) + (\FaceSize*0.10,\FaceSize, \FaceSize/2)$) circle (2pt);
\filldraw[gray, thick, rotate around={-45:(\xmid+\FaceSize*0.10, 2+\FaceSize, \FaceSize/2)}] 
  (\xmid+\FaceSize*0.10, 2+\FaceSize, \FaceSize/2) ellipse (6.5pt and 11pt);
\FACE{\xmid}{0}{0}{\FaceSize}{redface}
\filldraw[gray, thick, rotate around={-45:(\xmid+\FaceSize*0.10, 0+\FaceSize, \FaceSize/2)}] 
  (\xmid+\FaceSize*0.10, 0+\FaceSize, \FaceSize/2) ellipse (6.5pt and 11pt);
\node at ($(\xmid,0,0) + (\FaceSize*0.10,\FaceSize, \FaceSize/2)$) {\LARGE $\star$};
\FACE{\xmid}{-2}{0}{\FaceSize}{blueface}

\FACE{\xright}{2}{0}{\FaceSize}{greenface}
\FACE{\xright}{0}{0}{\FaceSize}{redface}
\filldraw[gray, thick, rotate around={-45:(\xright+\FaceSize*0.10, 0+\FaceSize, \FaceSize/2)}]
  (\xright+\FaceSize*0.10, 0+\FaceSize, \FaceSize/2) ellipse (6.5pt and 11pt);
\FACE{\xright}{-2}{0}{\FaceSize}{blueface}
\filldraw[gray, thick, rotate around={-45:(\xright+\FaceSize*0.10, 2+\FaceSize, \FaceSize/2)}] 
  (\xright+\FaceSize*0.10, 2+\FaceSize, \FaceSize/2) ellipse (6.5pt and 11pt);

\draw[dashed,-{Latex[length=1.25mm, width=2mm]},thick] 
  ($(\xmid,0,0) + (\FaceSize*0.10,\FaceSize, \FaceSize/2)$) -- 
  ($(\xmid,2,0) + (\FaceSize*0.10,\FaceSize, \FaceSize/2)$);
\draw[dashed,-{Latex[length=1.25mm, width=2mm]},thick] 
  ($(\xmid,0,0) + (\FaceSize*0.10,\FaceSize, \FaceSize/2)$) -- 
  ($(\xleft+0.1,0,0) + (\FaceSize*0.10,\FaceSize, \FaceSize/2)$);
\draw[dashed,-{Latex[length=1.25mm, width=2mm]},thick] 
  ($(\xmid,0,0) + (\FaceSize*0.10,\FaceSize, \FaceSize/2)$) -- 
  ($(\xright,0,0) + (\FaceSize*0.10-0.1,\FaceSize, \FaceSize/2)$);

%
\draw[dashed,-{Latex[length=1.25mm, width=2mm]},thick] 
  ($(\xmid,0,0) + (\FaceSize*0.10,\FaceSize, \FaceSize/2)$) -- 
  ($(\xleft+0.1,2,0) + (\FaceSize*0.10,\FaceSize, \FaceSize/2)$);
\draw[dashed,-{Latex[length=1.25mm, width=2mm]},thick] 
  ($(\xmid,0,0) + (\FaceSize*0.10,\FaceSize, \FaceSize/2)$) -- 
  ($(\xright,2,0) + (\FaceSize*0.10-0.1,\FaceSize, \FaceSize/2)$);

\node[anchor=south, text=gray] at ($(\xright+0.5,2.25,0) + (\FaceSize*0.10, \FaceSize, \FaceSize/2)$) {\Large $\bm{y}$};

%

\node[] at (0.5, 6.5, 0) {\Large $1$};
\node[] at (1.6, 6.5, 0) {\Large $2$};
\node[] at (2.7, 6.5, 0) {\Large $3$};

\end{tikzpicture}}
%
\caption{A local minimum $\bm{{\star}}=\left(\text{Red}, 2, \x^{\continuous}_{\star}\right)$ for the RBG example, where $\x^{\continuous} \in \mathcal{X}^{\continuous} \subset \mathbb{R}^2$, $m=1$ and $\eta=1$. 
}
\label{fig:min_complete}
\end{figure}

The local minimum, as described in \Cref{def:complete_local_minimum}, is not covered by \catmads.
In practice, the evaluations required for reaching optimality could be substantial.
%
%
%
%
To balance computational costs and solution quality, three weaker types of local minima are proposed.
%
%

The new types of local minima are named with a specific nomenclature in which the types of variables are combined by a hyphen 
or separated by a comma.
A hyphen symbolizes that the types of variables are jointly interacting.
For instance, the nomenclature of the local minimum in \Cref{def:complete_local_minimum} is \big[$\cat-\integer-\continuous$\big], since all types of variables are jointly interacting via Cartesian products.
The weaker types of local minima discard some or all of these joint interactions, meaning commas appear in their nomenclature.
Notably, the weakest type of local minima is named [$\cat$, $\integer$, $\continuous$], which signifies that all types of variables are separated.
%
Inspired by~\cite{ToNaTrCa2024}, the \big[$\cat$, $\integer$, $\continuous$\big]  minimum is defined below.

\begin{definition}[{\big[}$\cat$, $\integer$, $\continuous${\big]} local minimum]
A point $\x=\left(\x^{\cat}, \x^{\integer}, \x^{\continuous} \right) \in \Omega$ is said to be a {\em \big[$\cat$, $\integer$, $\continuous$\big] local minimum} if there exists a triplet $m \in \mathbb{N}^*$, $\eta \in \mathbb{N}^*$, $\epsilon \in \mathbb{R}_+^*$, such that

%

\[
\begin{array}{l@{\quad}l@{\;}l@{\quad}l}        f(\x) \leq f(\y), & \text{for all } \y^{\cat} &\in \mathcal{N}\left(\x^{\cat}; m\right) &\text{with } \y = \left(\y^{\cat}, \x^{\integer}, \x^{\continuous} \right) \in \Omega, \\
     f(\x) \leq f(\z), & \text{for all } \z^{\integer} &\in \mathcal{I}\left(\x^{\integer}; \eta\right) &\text{with } \z = \left(\x^{\cat}, \z^{\integer}, \x^{\continuous} \right) \in \Omega, \\
     f(\x) \leq f(\bm{u}), & \text{for all } \bm{u}^{\continuous} &\in \mathcal{B}\left(\x^{\continuous}; \epsilon\right) &\text{with } \bm{u} = \left(\x^{\cat},\x^{\integer},\bm{u}^{\continuous}\right) \in \Omega.
\end{array}
\]

%
%
     
%

\label{def:seperated_local_minimum}
\end{definition}

In \Cref{def:seperated_local_minimum}, each inequality is associated with a specific type of variables, meaning that there is no interaction between different types of variables.
The \big[$\cat$, $\integer$, $\continuous$\big] minimum is visualized for the RBG example in \Cref{subfig:min_component}.
For this type of local minimum, the vertical dotted line represents the point in the distance-induced neighborhood, whereas the horizontal dotted lines represent the integer neighborhood. 
%
%

\begin{figure}[htb!]
%
\begin{subfigure}[t]{0.32\textwidth}
\centering
\scalebox{0.85}{\begin{tikzpicture}

\clip (-1.55, -0.1) rectangle (3.1, 7);

\newcommand{\FACE}[5]{ 
  \node[] (C) at (#1, #2, #3) {}; 
  \fill[#5,opacity=0.4] 
    ($(C) + (#4*0.20,#4,#4)$) -- 
    ($(C) + (0,#4,#4)$) -- 
    ($(C) + (0,#4,0)$) -- 
    ($(C) + (#4*0.20,#4,0)$) -- 
    cycle;
  \draw[-] ($(C) + (#4*0.20,#4,#4)$) -- ($(C) + (0,#4,#4)$);
  \draw[-] ($(C) + (#4*0.20,#4,#4)$) -- ($(C) + (#4*0.20,#4,0)$);
  \draw[-] ($(C) + (0,#4,0)$) -- ($(C) + (0,#4,#4)$);
  \draw[-] ($(C) + (0,#4,0)$) -- ($(C) + (#4*0.20,#4,0)$);
}

\newcommand{\FaceSize}{4} 

\definecolor{blueface}{RGB}{51,171,255}
\definecolor{redface}{RGB}{236,87,87}
\definecolor{greenface}{RGB}{112,219,112}

\def\xleft{0}
\def\xmid{1.15}
\def\xright{2.3}

\FACE{\xleft}{2}{0}{\FaceSize}{greenface}
\FACE{\xleft}{0}{0}{\FaceSize}{redface}
\FACE{\xleft}{-2}{0}{\FaceSize}{blueface}
\fill[black] ($(\xleft,0,0) + (\FaceSize*0.10,\FaceSize, \FaceSize/2)$) circle (2pt);

\FACE{\xmid}{2}{0}{\FaceSize}{greenface}
\fill[black] ($(\xmid,2,0) + (\FaceSize*0.10,\FaceSize, \FaceSize/2)$) circle (2pt);
\FACE{\xmid}{0}{0}{\FaceSize}{redface}
\filldraw[gray, thick, rotate around={-45:(\xmid+\FaceSize*0.10, 0+\FaceSize, \FaceSize/2)}] 
  (\xmid+\FaceSize*0.10, 0+\FaceSize, \FaceSize/2) ellipse (6.5pt and 11pt);
\node at ($(\xmid,0,0) + (\FaceSize*0.10,\FaceSize, \FaceSize/2)$) {\LARGE $\star$};
\FACE{\xmid}{-2}{0}{\FaceSize}{blueface}

\FACE{\xright}{2}{0}{\FaceSize}{greenface}
\FACE{\xright}{0}{0}{\FaceSize}{redface}
\fill[black] ($(\xright,0,0) + (\FaceSize*0.10,\FaceSize, \FaceSize/2)$) circle (2pt);
\FACE{\xright}{-2}{0}{\FaceSize}{blueface}

\draw[dashed,-{Latex[length=1.25mm, width=2mm]},thick] 
  ($(\xmid,0,0) + (\FaceSize*0.10,\FaceSize, \FaceSize/2)$) -- 
  ($(\xmid,2,0) + (\FaceSize*0.10,\FaceSize-0.1, \FaceSize/2)$);

\draw[dashed,-{Latex[length=1.25mm, width=2mm]},thick] 
  ($(\xmid,0,0) + (\FaceSize*0.10,\FaceSize, \FaceSize/2)$) -- 
  ($(\xleft,0,0) + (\FaceSize*0.10,\FaceSize, \FaceSize/2)$);

\draw[dashed,-{Latex[length=1.25mm, width=2mm]},thick] 
  ($(\xmid,0,0) + (\FaceSize*0.10,\FaceSize, \FaceSize/2)$) -- 
  ($(\xright,0,0) + (\FaceSize*0.10-0.1,\FaceSize, \FaceSize/2)$);

\node[anchor=south] at ($(\xmid+0.16,2.0,0) + (\FaceSize*0.10,\FaceSize, \FaceSize/2)$) {\Large $\y$};
\node[anchor=east, text=gray] at ($(\xmid-0.15,-0.4,0) + (\FaceSize*0.10,\FaceSize, \FaceSize/2)$) {\Large $\bm{u}$};
\node[anchor=west] at ($(\xright,0,0) + (\FaceSize*0.10,\FaceSize, \FaceSize/2)$) {\Large $\z$};

\node[] at (0.5, 6.5, 0) {\Large $1$};
\node[] at (1.6, 6.5, 0) {\Large $2$};
\node[] at (2.7, 6.5, 0) {\Large $3$};

\end{tikzpicture}}
%
\subcaption{
\big[$\cat$, $\integer$, $\continuous$\big] local minimum.
}
\label{subfig:min_component}
\end{subfigure}
%
\captionsetup[subfigure]{justification=centering}
\begin{subfigure}[t]{0.32\textwidth}
\centering
\scalebox{0.85}{\begin{tikzpicture}

\clip (-1.55, -0.1) rectangle (3.1, 7);

\newcommand{\FACE}[5]{ 
  \node[] (C) at (#1, #2, #3) {}; 
  \fill[#5,opacity=0.4] 
    ($(C) + (#4*0.20,#4,#4)$) -- 
    ($(C) + (0,#4,#4)$) -- 
    ($(C) + (0,#4,0)$) -- 
    ($(C) + (#4*0.20,#4,0)$) -- 
    cycle;
  \draw[-] ($(C) + (#4*0.20,#4,#4)$) -- ($(C) + (0,#4,#4)$);
  \draw[-] ($(C) + (#4*0.20,#4,#4)$) -- ($(C) + (#4*0.20,#4,0)$);
  \draw[-] ($(C) + (0,#4,0)$) -- ($(C) + (0,#4,#4)$);
  \draw[-] ($(C) + (0,#4,0)$) -- ($(C) + (#4*0.20,#4,0)$);
}

\newcommand{\FaceSize}{4} 

\definecolor{blueface}{RGB}{51,171,255}
\definecolor{redface}{RGB}{236,87,87}
\definecolor{greenface}{RGB}{112,219,112}

\def\xleft{0}
\def\xmid{1.15}
\def\xright{2.3}

\FACE{\xleft}{2}{0}{\FaceSize}{greenface}
\FACE{\xleft}{0}{0}{\FaceSize}{redface}
\FACE{\xleft}{-2}{0}{\FaceSize}{blueface}
\fill[black] ($(\xleft,0,0) + (\FaceSize*0.10,\FaceSize, \FaceSize/2)$) circle (2pt);

\FACE{\xmid}{2}{0}{\FaceSize}{greenface}
\fill[black] ($(\xmid,2,0) + (\FaceSize*0.10,\FaceSize, \FaceSize/2)$) circle (2pt);
\filldraw[gray, thick, rotate around={-45:(\xmid+\FaceSize*0.10, 2+\FaceSize, \FaceSize/2)}] 
  (\xmid+\FaceSize*0.10, 2+\FaceSize, \FaceSize/2) ellipse (6.5pt and 11pt);
\FACE{\xmid}{0}{0}{\FaceSize}{redface}
\filldraw[gray, thick, rotate around={-45:(\xmid+\FaceSize*0.10, 0+\FaceSize, \FaceSize/2)}] 
  (\xmid+\FaceSize*0.10, 0+\FaceSize, \FaceSize/2) ellipse (6.5pt and 11pt);
\node at ($(\xmid,0,0) + (\FaceSize*0.10,\FaceSize, \FaceSize/2)$) {\LARGE $\star$};
\FACE{\xmid}{-2}{0}{\FaceSize}{blueface}

\FACE{\xright}{2}{0}{\FaceSize}{greenface}
\FACE{\xright}{0}{0}{\FaceSize}{redface}
\fill[black] ($(\xright,0,0) + (\FaceSize*0.10,\FaceSize, \FaceSize/2)$) circle (2pt);
\FACE{\xright}{-2}{0}{\FaceSize}{blueface}

\draw[dashed,-{Latex[length=1.25mm, width=2mm]},thick] 
  ($(\xmid,0,0) + (\FaceSize*0.10,\FaceSize, \FaceSize/2)$) -- 
  ($(\xmid,2,0) + (\FaceSize*0.10,\FaceSize, \FaceSize/2)$);

\draw[dashed,-{Latex[length=1.25mm, width=2mm]},thick] 
  ($(\xmid,0,0) + (\FaceSize*0.10,\FaceSize, \FaceSize/2)$) -- 
  ($(\xleft+0.1,0,0) + (\FaceSize*0.10,\FaceSize, \FaceSize/2)$);

\draw[dashed,-{Latex[length=1.25mm, width=2mm]},thick] 
  ($(\xmid,0,0) + (\FaceSize*0.10,\FaceSize, \FaceSize/2)$) -- 
  ($(\xright,0,0) + (\FaceSize*0.10-0.1,\FaceSize, \FaceSize/2)$);

\node[anchor=south, text=gray] at ($(\xmid+0.5,2.25,0) + (\FaceSize*0.10,\FaceSize, \FaceSize/2)$) {\Large $\bm{y}$};
\node[anchor=west] at ($(\xright,0,0) + (\FaceSize*0.10,\FaceSize, \FaceSize/2)$) {\Large $\z$};

\node[] at (0.5, 6.5, 0) {\Large $1$};
\node[] at (1.6, 6.5, 0) {\Large $2$};
\node[] at (2.7, 6.5, 0) {\Large $3$};

\end{tikzpicture}}
%
\subcaption{
\big[$\cat-\continuous$, $\integer$\big] local minimum.
}
\label{subfig:min_partially}
\end{subfigure}
%
\begin{subfigure}[t]{0.32\textwidth}
\centering
%
\scalebox{0.85}{\begin{tikzpicture}

\clip (-1.55, -0.1) rectangle (3.1, 7);

\newcommand{\FACE}[5]{ 
  \node[] (C) at (#1, #2, #3) {}; 
  \fill[#5,opacity=0.4] 
    ($(C) + (#4*0.20,#4,#4)$) -- 
    ($(C) + (0,#4,#4)$) -- 
    ($(C) + (0,#4,0)$) -- 
    ($(C) + (#4*0.20,#4,0)$) -- 
    cycle;
  \draw[-] ($(C) + (#4*0.20,#4,#4)$) -- ($(C) + (0,#4,#4)$);
  \draw[-] ($(C) + (#4*0.20,#4,#4)$) -- ($(C) + (#4*0.20,#4,0)$);
  \draw[-] ($(C) + (0,#4,0)$) -- ($(C) + (0,#4,#4)$);
  \draw[-] ($(C) + (0,#4,0)$) -- ($(C) + (#4*0.20,#4,0)$);
}

\newcommand{\FaceSize}{4}

\definecolor{blueface}{RGB}{51,171,255}
\definecolor{redface}{RGB}{236,87,87}
\definecolor{greenface}{RGB}{112,219,112}

\def\xleft{0}
\def\xmid{1.15}
\def\xright{2.3}

\FACE{\xleft}{2}{0}{\FaceSize}{greenface}
\FACE{\xleft}{0}{0}{\FaceSize}{redface}
\FACE{\xleft}{-2}{0}{\FaceSize}{blueface}
\fill[black] ($(\xleft,0,0) + (\FaceSize*0.10,\FaceSize, \FaceSize/2)$) circle (2pt);
\fill[black] ($(\xleft,2,0) + (\FaceSize*0.10,\FaceSize, \FaceSize/2)$) circle (2pt);

\FACE{\xmid}{2}{0}{\FaceSize}{greenface}
\fill[black] ($(\xmid,2,0) + (\FaceSize*0.10,\FaceSize, \FaceSize/2)$) circle (2pt);
\filldraw[gray, thick, rotate around={-45:(\xmid+\FaceSize*0.10, 2+\FaceSize, \FaceSize/2)}] 
  (\xmid+\FaceSize*0.10, 2+\FaceSize, \FaceSize/2) ellipse (6.5pt and 11pt);
\FACE{\xmid}{0}{0}{\FaceSize}{redface}
\filldraw[gray, thick, rotate around={-45:(\xmid+\FaceSize*0.10, 0+\FaceSize, \FaceSize/2)}] 
  (\xmid+\FaceSize*0.10, 0+\FaceSize, \FaceSize/2) ellipse (6.5pt and 11pt);
\node at ($(\xmid,0,0) + (\FaceSize*0.10,\FaceSize, \FaceSize/2)$) {\LARGE $\star$};
\FACE{\xmid}{-2}{0}{\FaceSize}{blueface}

\FACE{\xright}{2}{0}{\FaceSize}{greenface}
\FACE{\xright}{0}{0}{\FaceSize}{redface}
\fill[black] ($(\xright,0,0) + (\FaceSize*0.10,\FaceSize, \FaceSize/2)$) circle (2pt);
\FACE{\xright}{-2}{0}{\FaceSize}{blueface}
\fill[black] ($(\xright,2,0) + (\FaceSize*0.10,\FaceSize, \FaceSize/2)$) circle (2pt);

\draw[dashed,-{Latex[length=1.25mm, width=2mm]},thick] 
  ($(\xmid,0,0) + (\FaceSize*0.10,\FaceSize, \FaceSize/2)$) -- 
  ($(\xmid,2,0) + (\FaceSize*0.10,\FaceSize, \FaceSize/2)$);
\draw[dashed,-{Latex[length=1.25mm, width=2mm]},thick] 
  ($(\xmid,0,0) + (\FaceSize*0.10,\FaceSize, \FaceSize/2)$) -- 
  ($(\xleft+0.1,0,0) + (\FaceSize*0.10,\FaceSize, \FaceSize/2)$);
\draw[dashed,-{Latex[length=1.25mm, width=2mm]},thick] 
  ($(\xmid,0,0) + (\FaceSize*0.10,\FaceSize, \FaceSize/2)$) -- 
  ($(\xright,0,0) + (\FaceSize*0.10-0.1,\FaceSize, \FaceSize/2)$);

\draw[dashed,-{Latex[length=1.25mm, width=2mm]},thick] 
  ($(\xmid,0,0) + (\FaceSize*0.10,\FaceSize, \FaceSize/2)$) -- 
  ($(\xleft+0.1,2,0) + (\FaceSize*0.10,\FaceSize, \FaceSize/2)$);
\draw[dashed,-{Latex[length=1.25mm, width=2mm]},thick] 
  ($(\xmid,0,0) + (\FaceSize*0.10,\FaceSize, \FaceSize/2)$) -- 
  ($(\xright,2,0) + (\FaceSize*0.10-0.1,\FaceSize, \FaceSize/2)$);

\node[anchor=south, text=gray] at ($(\xmid+0.5,2.25,0) + (\FaceSize*0.10,\FaceSize, \FaceSize/2)$) {\Large $\bm{y}$};
\node[anchor=west] at ($(\xright,2,0) + (\FaceSize*0.10,\FaceSize, \FaceSize/2)$) {\Large $\z$};

\node[] at (0.5, 6.5, 0) {\Large $1$};
\node[] at (1.6, 6.5, 0) {\Large $2$};
\node[] at (2.7, 6.5, 0) {\Large $3$};

\end{tikzpicture}}
%
\subcaption{\big[$\cat-\continuous$, $\cat-\integer$\big] local minimum.
}
\label{subfig:min_mixed}
\end{subfigure}
\caption{Three new types of local minima at  $\bm{{\star}}=\left(\text{Red}, 2, \x^{\continuous}_{\star}\right)$ for the RBG example, where $\x^{\continuous} \in \mathcal{X}^{\continuous} \subset \mathbb{R}^2$, $m=1$ and $\eta=1$.
}
\label{fig:min_all}
\end{figure}
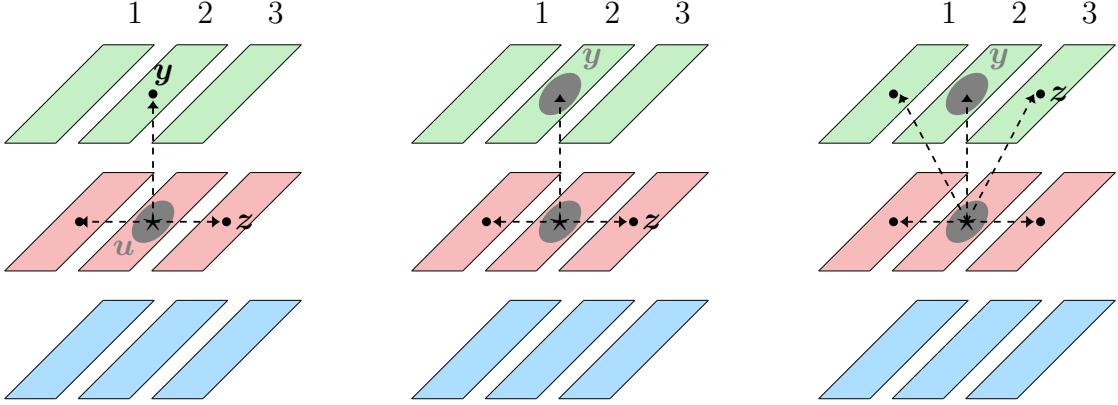

The next type of local minimum is \big[$\cat-\continuous$, $\integer$\big] with joint interactions between categorical and continuous variables.

\begin{definition}[{[}$\cat-\continuous$, int{]} local minimum]
A point $\x=\left(\x^{\cat}, \x^{\integer}, \x^{\continuous} \right) \in \Omega$ is said to be a {\em \big[$\cat-\continuous$, $\integer$\big] local minimum} if there exists a triplet $m \in \mathbb{N}^*$, $\eta \in \mathbb{N}^*$, $\epsilon \in \mathbb{R}^*_+$, such that
\[
\begin{array}{l@{\ }l@{\,}c@{\;}l@{\,}l}         f(\x) \leq f(\y), &\text{for all } &\left(\y^{\cat}, \y^{\continuous}\right)  &\in \mathcal{N}\left(\x^{\cat}; m\right) \times \mathcal{B}\left(\x^{\continuous}; \epsilon\right)  & \text{with} \ \y = \left(\y^{\cat}, \x^{\integer}, \y^{\continuous} \right) \in \Omega,    \\
   f(\x) \leq f(\z), &\text{for all } &\z^{\integer} &\in  \mathcal{I}\left(\x^{\integer}; \eta\right) & \text{with} \ \z = \left(\x^{\cat}, \z^{\integer}, \x^{\continuous} \right) \in \Omega. 
\end{array}
\]

%
   %
    %
   %
%

\label{def:partially_local_minimum}
\end{definition}

The \big[$\cat-\continuous$, $\integer$\big] local minimum type is schematized for the RBG example in \Cref{subfig:min_partially}.
In comparison to the \big[$\cat$, $\integer$, $\continuous$\big] local minimum type,
the \big[$\cat-\continuous$, $\integer$\big] type strengthens optimality with a joint interaction in the first inequality. 
In \Cref{subfig:min_partially}, this interaction
is visualized at $\y$, which has a continuous ball around it.

The last type of local minimum is
\big[$\cat-\continuous$, $\cat-\integer$\big] with interactions between the categorical and continuous variables, and between the categorical and integer variables.
\begin{definition}[{[}$\cat-\continuous$, $\cat-\integer${]} local minimum]
A point $\x=\left(\x^{\cat}, \x^{\integer}, \x^{\continuous} \right) \in \Omega$ is said to be a {\em \big[$\cat-\continuous$, $\cat-\integer$\big] local minimum} if there exists a triplet $m \in \mathbb{N}^*$, $\eta \in \mathbb{N}^*$, $\epsilon \in \mathbb{R}_+^*$, such that
\[
\begin{array}{l@{\ }l@{\,}c@{\;}l@{\,}l}          f(\x) \leq f(\y), &\text{for all } &\left(\y^{\cat}, \y^{\continuous}\right) &\in \mathcal{N}\left(\x^{\cat}; m\right) \times \mathcal{B}\left(\x^{\continuous}; \epsilon\right)  & \text{with} \ \y = \left(\y^{\cat}, \x^{\integer}, \y^{\continuous} \right) \in \Omega,   \\
    f(\x) \leq f(\z), &\text{for all } &\left(\z^{\cat}, \z^{\integer} \right) \ &\in  \mathcal{N}\left(\x^{\cat}; m\right) \times \mathcal{I}\left(\x^{\integer}; \eta\right) & \text{with} \ \z = \left(\z^{\cat}, \z^{\integer}, \x^{\continuous} \right) \in \Omega.
\end{array}
\]

    %
%

%

\label{def:mixed_local_minimum}
\end{definition}
The \big[$\cat-\continuous$, $\cat-\integer$\big] minimum further strengthens the optimality of the \big[$\cat-\continuous$, $\integer$\big] type by adding interactions between categorical and integer variables.
%
%
In \Cref{subfig:min_mixed}, this additional interaction is represented at $\z$, which modifies the categorical and integer variables of $\x_{\star}$.

\subsection{Refining subsequence}
\label{sec:analysis_refining_subsequence}

This part of the convergence analysis studies the behavior of sequences of iterates produced by \catmads.
%
%
As for \gmads ~\cite{AuLeDTr2018}, convergence is established via unsuccessful iteration subsequences in which the step sizes of continuous variables get infinitely fine, and step sizes of integer variables reduce to the value one.
For \catmads, these subsequences must also possess convergent distance-induced neighborhoods. 
As a generalization of \gmads, the convergence of the step and frame size parameters is directly inherited from \gmads. 
%
%
This leads to this key definition of the convergence analysis.
%
\begin{definition}[Refining subsequence and refined point]
Let $U$ be the set of indices of the unsuccessful iterations generated by an instance of \catmads. 
If $K$ is an infinite subset of $U$ such~that
%
         %
         %
         %
%
         %
         %
         %
%
\[
\begin{array}{l@{\,}l@{\,}l@{\ }l@{\,}} 
        &\lim \limits_{k \in K} \x_{(k)} &=& \x_{\star} \qquad\qquad\quad  \text{ for some } \x_{\star} \in \mathcal{X},  \\[0.3cm]
         &\lim \limits_{k \in K} \| \bm{\delta}_{(k)}  \|_2^2 &=& n^{\integer}, \\[0.3cm]
         &\lim \limits_{k \in K} m_{(k)} &=& m_{\star}, \\[0.3cm]
         &\lim \limits_{k \in K} \mathcal{N} \left( \x_{(k)}^{\cat}; m_{(k)} \right)  &=& \mathcal{N}\left(\x_{\star}^{\cat}; m_{\star} \right),
\end{array}
\]
then the subsequence $\{ \x_{(k)} \}_{k\in K}$ of poll centers with $\x_{(k)} \in M_{(k)}$ is said to be a {\em refining subsequence}, and $\x_{\star} \in \mathcal{X}$ is said to be a {\em refined point}. 
\label{def:refining}
\end{definition}
%
In \Cref{def:refining}, the notation $\lim_{k \in K} \| \bm{\delta}_{(k)}  \|_2^2 = n^{\integer}$ is equivalent to $\lim_{k\in K} \delta_{(k),i}^{\integer}=1 $ for all $ i \in I^{\integer}$ and $\lim_{k\in K} \delta_{(k),j}^{\continuous}=0 $ for all $ j \in I^{\continuous}$.
%
%
Moreover, it is the PB that determines if a refining subsequence converges to a feasible refined point or not.
From \Cref{def:refining}, the next \textit{zeroth-order} result of~\Cref{thm:zeroth} is proposed with only the two assumptions. 
Since \catmads is built upon \gmads , the proof of the next theorem is partly given by construction.
The rest of the proof sequentially chooses infinite subsequence that satisfy the requirements of a refining subsequence~\cite{AuDe2006}.
%
\begin{theorem}[Zeroth-order]
Let $\{ \x_{(k)} \}$
    be the sequence of points generated by \catmads
    from an initial point $\x_{(0)} \in \mathcal{X}$ with $f\left(\x_{(0)}\right) \in \mathbb{R}$.
If all iterates $\{ \x_{(k)} \}$ belong to a compact set, then
there exists a refining subsequence $\{ \x_{(k)} \}_{k \in K}$ with corresponding refined point $\x_{\star} \in \mathcal{X}$.

\label{thm:zeroth}
\end{theorem}
\begin{proof}
Let $U$ be the set of indices of unsuccessful iterations generated by \catmads.
For all $k>0$, the intersection of a mesh $M_{(k)}$ and the compact domain $\mathcal{X}$ is finite.
Thus, relative to \gmads , the extended poll introduces only a finite sequence of quantitative polls per iteration, and the categorical poll introduces only a finite number of categorical modifications possible per iteration.
%
%
Therefore, \catmads inherits the following result of \gmads : a sequence of iterates $\{x_{(k)}\}$ generated in a compact set is such that $\liminf_{k \rightarrow \infty} \delta_{(k),i}^{\continuous}=0 \, \text{for each } i \in I^{\continuous}$ and $\liminf_{k \rightarrow \infty} \delta_{(k),j}^{\integer}=0 \, \text{for each } j \in I^{\integer}$ (see Theorems 3.1 and 3.2 in~\cite{AuLeDTr2018}).
%
It follows that there exists an infinite subset of indices $K_1 \subseteq U$ such that $\lim_{k \in K_1} \| \bm{\delta}_{(k)}  \|_2^2 = n^{\integer}$.
Similarly, since the set is compact, there exists an infinite subset of indices $K_2 \subseteq K_1$ such that $\lim_{k \in K_2} \x_{(k)^{\continuous}} = \x_{\star}^{\continuous}$.
Additionally, since the categorical variables, integer variables and number of neighbors belong to finite sets, infinite subsets of indices can be chosen successively as $K_5 \subseteq K_4, \subseteq K_3$, such that $\lim_{k \in K_3} \x_{(k)}^{\integer} = \x_{\star}^{\integer}$, $\lim_{k \in K_4} \x_{(k)}^{\cat} = \x_{\star}^{\cat}$ and $\lim_{k \in K_5} m_{(k)} = m_{\star} $.
Consequently, $\lim_{k \in K_5} \mathcal{N} \left( \x_{(k)}^{\cat}; m_{(k)} \right)  = \mathcal{N}\left(\x_{\star}^{\cat}; m_{\star} \right)$.
Finally, the infinite subset of indices $K=K_5$ satisfies \cref{def:refining}.
\end{proof}

Note that implementation-wise, converging distance-induced neighborhoods in a refining subsequence can be easily ensured by fixing both the number of neighbors $m_{(k)}$ and the categorical distance $d^{\cat}$ after a given iteration $k>0$.

\subsection{Nonsmooth calculus for \catmads}
\label{sec:analysis_calculus}

The second part of the convergence analysis involves the continuous variables for refined points.
%
For fixed categorical and integer variables, the analysis involves the Clarke generalized derivative~\cite{Clar83a, Jahn2007}
and the hypertangent cone~\cite{Rock70a}, restricted to the continuous variables.
The vectors $\bm{u}, \bm{v}$ and $\bm{w}$ are used next to denote continuous vectors.
%
%
The \textit{continuous feasible set at the refined point} $\x_{\star}$, which contains the continuous components that are feasible given the categorical and integer components of $\x_{\star}$, is defined as 
\begin{equation}
    \Lambda_{\star} \coloneq \Big\{ \bm{u} \in \mathcal{X}^{\continuous} \, : \, \left(\x_{\star}^{\cat}, \x_{\star}^{\integer}, \bm{u} \right) \in \Omega \Big\}.
\label{eq:feasible_continuous_set}
\end{equation}
For a given refined point $\x_{\star}$,  the following definition is such that the hypertangent cone contains vectors in $\mathbb{R}^{n^{\continuous}}$ and rooted at $\x_{\star}^{\continuous}$.
%
%
%
\begin{definition}[Hypertangent cone~\cite{Rock70a}]
A vector $\bm{v} \in \mathbb{R}^{n^{\continuous}}$ is said to be a {\em hypertangent vector} to the set $\Lambda_{\star}$ at a point $\x_{\star}= \left( \x_{\star}^{\cat}, \x_{\star}^{\integer}, \x_{\star}^{\continuous} \right)\in \Omega$, if there exists a scalar $\epsilon>0$ such that 
\begin{equation*}
\left( \x^{\cat}, \x^{\integer}, \bm{w}+t \bm{u} \right) \in \Omega \
\text{ for all } \ \bm{w} \in \mathcal{B} \left( \x_{\star}^{\continuous}; \epsilon\right) \cap \ \Lambda_{\star}, \   \bm{u} \in \mathcal{B} \left( \bm{v}; \epsilon \right) \  \text{ and } \ t \in (0, \epsilon).
\end{equation*}
\noindent The {\em hypertangent cone} $T_{\Lambda_{\star}}^{H}(\x_{\star})$ to $\Lambda_{\star}$ at $\x_{\star} \in \Omega$ is the set of all hypertangent vectors to $\Lambda_{\star}$ at $x_{\star}$.
\label{def:hypertangent_cone}
\end{definition}
In \Cref{subfig:analysis_HT}, the continuous feasible set $\Lambda_{\star}$ and the hypertangent cone $T_{\Lambda_{\star}}^{H}$ at a refined point $\x_{\star}$ are illustrated for the RBG example, where $x_{\star}^{\integer}=2$.
\begin{figure}[htb!]
\begin{subfigure}[t]{0.45\textwidth}
\centering
  \scalebox{0.85}{\begin{tikzpicture}[use Hobby shortcut]

\clip (0, 1) rectangle (5.5, 8.5);

   \newcommand{\FACE}[5]{ 
    \node[] (C) at (#1, #2, #3) {}; 
    \fill[#5,opacity=0.4] ($(C) + (#4*0.4,#4,#4)$) -- ($(C) + (0,#4,#4)$) -- ($(C) + (0,#4,0)$) -- ($(C) + (#4*0.4,#4,0)$) -- cycle; 
    \draw[-] ($(C) + (#4*0.4,#4,#4)$) -- ($(C) + (0,#4,#4)$); 
    \draw[-] ($(C) + (#4*0.4,#4,#4)$) -- ($(C) + (#4*0.4,#4,0)$); 
    \draw[-] ($(C) + (0,#4,0)$) -- ($(C) + (0,#4,#4)$); 
    \draw[-] ($(C) + (0,#4,0)$) -- ($(C) + (#4*0.4,#4,0)$); 
    }
    \newcommand{\FaceSize}{6} 

    \definecolor{blueface}{RGB}{51,171,255}   
    \definecolor{redface}{RGB}{236,87,87}     
    \definecolor{greenface}{RGB}{112,219,112} 


    \newcommand{\curve}{
      (-2.2,-1)
      .. (0.2,-1.5)
      .. (0.8,-1.7) 
      .. (2,2)
      .. (0,1)
    }
    
    \FACE{3}{2.5}{0}{\FaceSize}{greenface}

    \begin{scope}[shift={(-0.15+2.9, 5+2.25)}, scale=0.75, rotate=30]]
        \draw[closed, scale=0.4, fill=gray,gray] \curve;
    \end{scope}

    \FACE{3}{0}{0}{\FaceSize}{redface}
    \begin{scope}[shift={(3.25,5)}]
        \coordinate (a) at (-1.06,-1.06); 
        \node [circle,draw,fill=gray,gray] (c) at (0,0) [minimum size=40pt] {};
        \draw[gray,fill] (a) -- (tangent cs:node=c,point={(a)},solution=1) --
        (c.center) -- (tangent cs:node=c,point={(a)},solution=2) -- cycle;
        
        \coordinate (tip1) at ($ (a) + (1.2,0.4) $);
        \coordinate (tip2) at ($ (a) + (0.4,1.2) $);
        
        \draw[-{Latex[length=1.25mm, width=2mm]},thick] (a) -- (tip1);
        \draw[-{Latex[length=1.25mm, width=2mm]},thick] (a) -- (tip2);
        
        \draw[dotted,line width=1.5pt] (tip1) to[out=90,in=0] (tip2);
        
        \node at ($ (a) + (0.65,0.6) $) {\Large $T^H_{\Lambda_*}$}; 
        \node at ($ (c) + (0.2,0.4) $) {\Large $\Lambda_*$}; 
    \end{scope}
    %
    \fill[black] ($(2.65,-0.9) + (\FaceSize*0.125,\FaceSize, \FaceSize/2)$) circle (3pt);
    \node[] at ($(2.2,-0.9) + (\FaceSize*0.125,\FaceSize, \FaceSize/2)$) {\Large $\x_{\star}$};

    \FACE{3}{-2.5}{0}{\FaceSize}{blueface}

    \newcommand{\curvee}{(-1,0) .. (0.5,3) .. (2,2) .. (0,1)}
    \begin{scope}[shift={(6-2.9, 4.25-2.5)}]
        \draw[closed, scale=0.4, fill=gray,gray] \curvee;
    \end{scope}

\end{tikzpicture}}
  \subcaption{Continuous feasible set and hypertangent cone at a refined point $\x_{\star}$, where gray areas represent feasible regions and $x_*^{\integer}=2$.}
  \label{subfig:analysis_HT}
\end{subfigure}
\hspace{0.02cm}
\begin{subfigure}[t]{0.45\textwidth}
\centering
  \scalebox{0.95}{\begin{tikzpicture}
    \definecolor{myblue}{RGB}{51,171,255}   
    \definecolor{myred}{RGB}{236,87,87}     
    \definecolor{mygreen}{RGB}{112,219,112} 

    \def\PlaneShape{(-2,-1) -- (4,-1) -- (5,2) -- (-1,2) -- cycle;}

    \fill[myred,opacity=0.4] \PlaneShape;
    \draw[black, thick] \PlaneShape;

    \begin{scope}

        \draw[thick, smooth] (-1.25, 1.25) .. controls (-1.25 + 0.25, 1.25 - 0.25) and (0+0.25+1, 2-0.25) .. (0+1, 2);
        \draw[thick, smooth] (-1.25-0.33, 1.25-1) .. controls (-1.25 + 0.25-0.33, 1.25 - 0.25-1) and (0+0.25+1+1.5, 2-0.25) .. (0+1+1.5, 2);

        \draw[thick] (-2, -1) -- (5, 2);

        \begin{scope}[shift={(3, 1)}]
            \draw[thick, smooth] (1.25, -1.25) .. controls (1.25 - 0.25, -1.25 + 0.25) and (0 - 0.25 - 1, -2 + 0.25) .. (0 - 1, -2);
        \end{scope}
    
        \begin{scope}[shift={(3, 1)}]
            \draw[thick, smooth] (1.25 + 0.33, -1.25 + 1) .. controls (1.25 - 0.25 + 0.33, -1.25 + 0.25 + 1) and (0 - 0.25 - 1 - 1.5, -2 + 0.25) .. (0 - 1 - 1.5, -2);
        \end{scope}

    \end{scope}

    \node[right] at (0, 1.5) {$4$};
    \node[right] at (1.25, 1.33) {$2$};
    \node[right] at (3, 1.25) {$1$};
    \node[right] at (3, 0.4) {$2$};
    \node[right] at (3, -0.33) {$4$};

    \def\shift{3.6} 
    
    \begin{scope}[shift={(0, \shift)}]
        \fill[mygreen,opacity=0.4] \PlaneShape;
        \draw[black, thick] \PlaneShape;
    \end{scope}

    \begin{scope}[shift={(0, \shift)}]
        \clip \PlaneShape;
        \foreach \y in {-0.5, 0, 0.5, 1, 1.5} { 
            \draw[black, thick] (-3,\y) -- (5,\y);
        }
    \end{scope}

    \begin{scope}[shift={(0, \shift)}]
        \node[right] at (-1-0.75, -0.5) {\large $5$};
        \node[right] at (-1-0.6, 0) {\large $4$};
        \node[right] at (-1-0.45, 0.5) {\large $3$};
        \node[right] at (-1-0.3, 1) {\large $2$};
        \node[right] at (-1-0.15, 1.5) {\large $3$};
    \end{scope}

    \begin{scope}[shift={(-0.25, -0.25)}]
        \node[below] at (0.25,-0.15) {$\x_{(k)}$};
        \fill[black] (-0.2, -0.5) circle (2pt);
        \node[above] at (0.2, \shift-0.8) {$\y_{(k)}$};
        \fill[black] (-0.2, \shift-0.5) circle (2pt);
        \draw[dashed,-{Latex[length=3mm, width=2mm]},thick] (-0.2, -0.5) -- (-0.2, \shift-0.5);

        \node[rotate=75] at (0.0, \shift+0.5)  {\Large $\ldots$};

          \begin{scope}[shift={(0.4, \shift-1.875)}] 
            \node[above] at (-0.15, \shift-0.55+0.15) {$\z_{(k)}$};
            \fill[black] (-0.2, \shift-0.5+0.15) circle (2pt);
        \end{scope}
    \end{scope}
    %

    \begin{scope}[shift={(-0.25, -0.25)}]
         \node[below] at (1.6, 0.5) {$\x_{(k+1)}$};
        \fill[black] (1.0, 0.15) circle (2pt);
        \node[above] at (1.6, \shift-0.3) {$\y_{(k+1)}$};
        \fill[black] (1.0, \shift-0.35 + 0.35) circle (2pt);
        \draw[dashed,-{Latex[length=3mm, width=2mm]},thick] (1.0, 0.15) -- (1.0, \shift-0.35 + 0.35);
         \begin{scope}[shift={(0.4, 1.15-0.25+0.075)}] 
            \node[above] at (1.0, 3.875+0.075) {$\z_{(k+1)}$};
            \fill[black] (1.0, 3.95) circle (2pt);
        \end{scope}
        %
        \node[rotate=75] at (1.25, \shift+0.75)  {\Large $\ldots$};
    \end{scope}

    \node[below] at (2.2+0.3, 0.1 + 0.7+0.4) {$\x_{\star}$};
    \fill[black] (2.2, 0.8) circle (2pt);
    \node[above] at (2.66, \shift) {$\y_*$};
    \fill[black] (2.2, \shift+0.5) circle (2pt);
    \draw[dashed,-{Latex[length=3mm, width=2mm]},thick] (2.2,0.1+0.8) -- (2.2, \shift+0.5);
     \begin{scope}[shift={(0.4, 1.5-0.5)}] 
        \node[above] at (2.2, \shift+0.05) {$\z_{\star}$};
        \fill[black] (2.2, \shift) circle (2pt);
    \end{scope}
    %
    \node[rotate=52.5] at (2.4, \shift+0.75)  {\Large $\ldots$};

    \path (1.0,-0.2) -- (2.2,0.1) coordinate[midway] (x_mid);
    \node[rotate=30] at ($(x_mid) + (0,0.5)$) {\Large $\ldots$};
    \path (1.0, \shift-0.35) -- (2.2, \shift) coordinate[midway] (y_mid);
    \node[rotate=30] at ($(y_mid) + (-0.1,0.3)$) {\Large $\ldots$};
    \path (1.0, 3.95) -- (2.2, \shift+0.05) coordinate[midway] (z_mid);
    \node[rotate=0] at ($(z_mid) + (0.25,0.8)$) {\Large $\ldots$};

\end{tikzpicture}}
  \subcaption{Refining, neighbor and endpoint subsequences for two categories example.}
  \label{subfig:analysis_extendedpoll}
\end{subfigure}
\caption{Illustrative examples for the convergence analysis.}
\label{fig:analysis}
\end{figure}
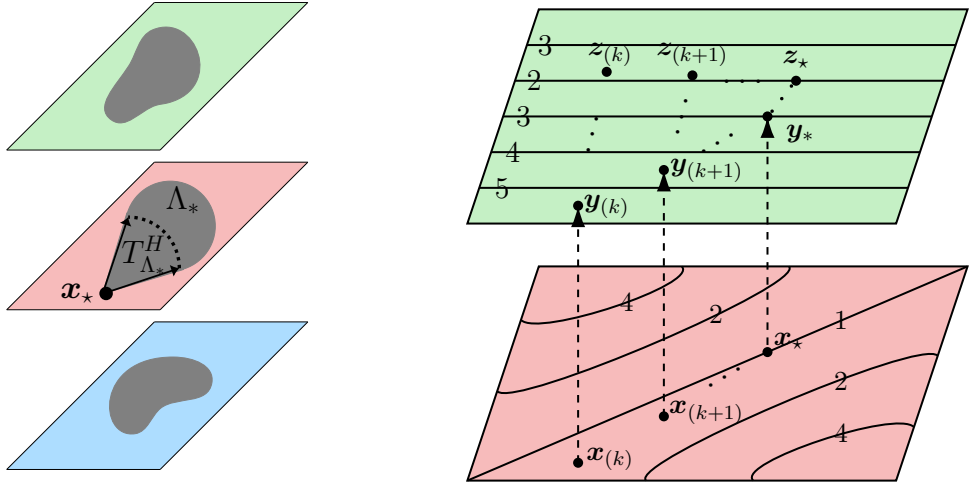
In the figure, the feasible regions are represented in gray.
The shape depends on the RBG category.
%
%
The hypertangent cone  $T_{\Lambda_{\star}}^{H} \subset \mathbb{R}^{2}$ is represented with a dotted curve and contains the continuous directions rooted at $\x_{\star}$ that point inside $\Lambda_{\star}$.
%
%
%
The next definition formalizes a subset of these continuous directions.
\begin{definition}[Refined direction~\cite{AuDe2006}]
Let $\{\x_{(k)} \}_{k \in K}$ be a refining subsequence with corresponding refined point $\x_{\star}$. A vector $\bm{v}_{\star} \in \mathbb{R}^{n^{\continuous}}$ is said to be a {\em refining direction} if there exists an infinite subsequence $L \subseteq K$ with poll directions $\bm{d}_{(k)} = \left(\bm{d}_{(k)}^{\integer}, \bm{d}_{(k)}^{\continuous} \right) \in \mathcal{D}_{(k)}$ such that $\bm{d}_{(k)}^{\integer} = \bm{0}$ and $\lim_{k \in L} \frac{\bm{v}_{(k)}}{\| \bm{v}_{(k)} \| } = \frac{\bm{v}_{\star}}{\| \bm{v}_{\star} \| }$ with $\bm{v}_{(k)} = \diag\left( \bm{\delta}_{(k)}^{\continuous} \right) \bm{d}^{\continuous}$.
\label{def:refined_direction}
\end{definition}
%
%
To employ generalized directional derivatives~\cite{Clar83a}, the objective function $f$ must be locally Lipschitz continuous with respect to the continuous variables.
For a given refined point $x_{\star}$, the restriction of the objective function in $\mathcal{X}^{\continuous}$ is noted $f_{\star} : \mathcal{X}^{\continuous} \to 
\overline{\mathbb{R}}$ and formulated as
\begin{equation}
%
    f_{\star} \left( \bm{u} \right) \coloneq f \left( \x_{\star}^{\cat}, \x_{\star}^{\integer}, \bm{u}  \right).
\label{eq:objective_continuous}
\end{equation}
{
With this function, the generalized directional derivative can be redefined for this work.

\begin{definition}[Generalized directional derivative~\cite{Clar83a, Jahn2007}]
For a refined point $\x_{\star} \in \Omega$, let $f_{\star}: \mathcal{X}^{\continuous} \to \overline{\mathbb{R}}$ be a Lipschitz continuous function near $\bm{u} \in \Lambda_{\star}$. 
The {\em generalized directional derivative} of the restriction $f_{\star}$ at $\bm{u}$ in the direction $\bm{v} \in \mathbb{R}^{n^{\continuous}}$ is defined by
\begin{equation}
   f^{\circ}_{\star} \left(\bm{u}; \bm{v} \right) \coloneq
    \limsup_{ \substack{\w \rightarrow \bm{u}, \, t\searrow 0, \\ \bm{w} \in \Lambda_{\star}, \, \bm{w}+t\bm{v}\in \Lambda_{\star} } } \frac{f_{\star}(\w+t\bm{v})-f_{\star}(\w)}{t}.
\end{equation}
\end{definition}

Under a local Lipschitz continuity assumption, the following theorem combines the definitions of hypertangent cone, refined direction and generalized directional derivative, to strengthen the convergence results of a refined point $\x_{star}$ with respect to its continuous variables.
}

\begin{theorem}[Generalized directional derivative at a refined point]
Let $\x_{\star} \in \Omega$ be a refined point generated by \catmads such that $f_{\star}$ is Lipschitz continuous near $\x_{\star}^{\continuous}$.
If $\bm{v}_{\star} \in T_{\Lambda_{\star}}^{H}(\x_{\star})$ is a refining direction, then $f^{\circ}_{\star} \left(\x_{\star}^{\continuous}; \bm{v}_{\star} \right) \geq 0$. 
\label{thm:clarke}
\end{theorem}
\begin{proof}
Let $\{ \x_{(k)} \}_{k \in L}$
be a refining subsequence with corresponding refined point $\x_{\star} \in \Omega$, and let $\bm{v}_{\star} \in T_{\Lambda_{\star}}^{H}(\x_{\star})$ be a refining direction.
%
%
From \Cref{def:refining}, the categorical and integer components are fixed, such that $\x^{\cat}_{(k)}= \x_{\star}^{\cat}$ and $\x^{\integer}_{(k)}= \x_{\star}^{\integer}$, ensuring that the continuous feasible set $\Omega_{\star}$ does not change through the iterations.
\Cref{def:quantitative_mesh} ensures the existence of bounds 
    $\bm{\Delta}_{(k)}^{\continuous} \geq \bm{0}$
    such that the continuous directions satisfy
    $- \bm{\Delta}_{(k)}^{\continuous} \leq \bm{v}_{(k)} = \diag\left( \bm{\delta}_{(k)}^{\continuous} \right) \bm{d}^{\continuous} \leq \bm{\Delta}_{(k)}^{\continuous}$ for all $k \in L$.
%
\Cref{thm:zeroth} ensures that $\lim_{k \in L} \| \bm{\Delta}_{(k)}^{\continuous} \| = 0$,
 and therefore, $\lim_{k \in L} \| \bm{v}_{(k)} \| = 0$.
%
%
By \Cref{def:refined_direction}, the refining direction $\bm{v}_{\star} \in  T_{\Lambda_{\star}}^{H}(\x_{\star})$ is such that $\lim_{k \in L} \frac{\bm{v}_{(k)}}{\| \bm{v}_{(k)} \| } = \frac{\bm{v}_{\star}}{\| \bm{v}_{\star} \| }$.
Thus, it follows that $\x_{(k)}^{\continuous}+\bm{v}_{(k)} \in \Lambda_{\star}$ for all $k \in L$ sufficiently large, by the definitions of the hypertangent cone and the refining direction, and knowing thats $\lim_{k \in L} \| \bm{v}_{(k)} \| = 0$.
It follows that $f^{\circ}_{\star} \left(\x_{\star}^{\continuous}; \bm{v}_{\star} \right) \geq 0$ can be developed exactly as in Theorem 6.4 of~\cite{AuLedTr2014}, since $\x_{(k)}^{\continuous}+\bm{v}_{(k)}$ are unsuccessful iterations with $\bm{d}_{(k)}^{\integer}=\bm{0} \ \text{for all } k \in L$.  
\end{proof}

The next result provides a necessary optimality condition
with generalized directional derivatives at a refined point, provided that the function $f_{\star}$ is locally Lipschitz continuous.
%
%
The result depends on whether the PB produces a feasible or infeasible refined point.
If the refined point is feasible, then the result involves the function $f_{\star}$ and the hypertangent cone $T_{\Lambda_{\star}}^{H}(\x_{\star})$, as \Cref{thm:clarke}.
Otherwise, it involves the function $h_{\star}$ (with the same expression of $f_{\star}$ in \Cref{{eq:objective_continuous}} but for the aggregation function $h$) and the continuous set $\mathcal{X}^{\continuous} \subseteq \Lambda_{\star}$: in this case, the optimization problem of interest is $\min_{\x \in \mathcal{X}} h(\x)$.

\begin{theorem}[Convergence of \catmads]
Let $\{ \x_{(k)}\}_{k \in K}$ be a refining subsequence with corresponding refined point $\x_{\star}$ produced by \catmads with the PB.
Suppose that the set of refining directions is dense in the unit sphere $\{ \bm{u} \in \mathbb{R}^{n^{\continuous}} \, : \,  \| \bm{u} \| = 1  \}$. 
%

   %
%
%
\[
\begin{array}{r@{\;}l@{\;}l@{\;}l@{\;}l}
\text{i)} &\text{If } \x_{\star} \in \Omega
   &\text{ and } f_{\star} \text{ is Lipschitz near } \x_{\star}^{\continuous}, 
   &\text{then } f^{\circ}_{\star} \left(\x_{\star}^{\continuous}; \bm{v} \right) \geq 0 
   &\text{for all } \bm{v} \in T_{\Lambda_{\star}}^{H}(\x_{\star}). \\
\text{ii)} &\text{If } \x_{\star} \in \mathcal{X} \setminus \Omega &\text{ and } h_{\star} \text{ is Lipschitz  near } \x_{\star}^{\continuous}, 
    &\text{then } h^{\circ}_{\star} \left(\x_{\star}^{\continuous}; \bm{v} \right) \geq 0
    &\text{for all } \bm{v} \in T_{\mathcal{X}^{\continuous}}^{H}(\x_{\star}). 
\end{array}
\]

\label{thm:convergence_catmads}
\end{theorem}
\begin{proof}
Since the refining subsequence $\{ \x_{(k)}\}_{k \in K}$ is such that $\x_{(k)}^{\cat}=\x_{\star}^{\cat}$ and $\x_{(k)}^{\integer}=\x_{\star}^{\integer} \, \text{for all } k \in K$, the proof is a direct consequence of \Cref{thm:clarke}, as well as  Corollary 3.4, Theorem 3.5. and Corollary 3.6 from~\cite{AuDe09a} (PB for the continuous case).
\end{proof}

\Cref{thm:convergence_catmads} is a fundamental theoretical result of \catmads and represents a necessary condition to achieve a separated local minimum with $\eta=1$ and $m=m_{\star} \in \mathbb{N}$.
%

\subsection{Stronger local minima}
\label{sec:analysis_interactions}

%
\Cref{def:refining} implies that a refining subsequence $\{\x_{(k)}\}_{k \in K}$ has an associated subsequence $\{\y_{(k)} \}_{k \in K}$, called the \textit{neighbor subsequence}, whose points 
have categorical components in the distance-induced neighborhoods of $\{\x_{(k)}\}_{k \in K}$,
\textit{i.e.}, $\smash{\y_{(k)}=\left(\y_{(k)}^{\cat}, \x_{(k)}^{\quant}\right) \in \mathcal{N} \left( \x_{(k)}^{\cat}; m_{(k)} \right) \times \mathcal{X}^{\quant}}$.
Since the refining subsequence converges to a refined point $\x_*$, the neighbor subsequence converges to a point $\y_{\star}$, referred to as a $\textit{refined neighbor point}$.
%
%
The next theorem, adapted from \mvmads ~\cite{AACW09a}, ensures that a refined point $\x_*$ has an objective value that is lower or equal than a neighbor refined point $\y_*$, under continuity hypothesis.
%
%
\begin{theorem}
If the restriction of the objective function in $\mathcal{X}^{\continuous}$ is lower semi-continuous at a refined point $\x_{\star}$ produced by \catmads, and upper semi-continuous at the refined neighbor point $\y_{\star}$, 
then $f(\x_{\star}) \leq f (\y_{\star})$.

\label{thm:neighbors}
\end{theorem}
\begin{proof}
Let $\{\x_{(k)}\}_{k \in K}$ be a refining subsequence converging to the refined point $\x_{\star}$ , and let $\{\y_{(k)} \}_{k \in K}$ be a neighbor subsequence converging to the neighbor refined point $\y_{\star}$.
By definition of the refining subsequence, it is ensured that $f( \x_{(k)}) \leq f\left(\y_{(k)}\right) \, \text{for all } k\in K$ since $\x_{(k)}$ yields an unsuccessful iteration and $\y_{(k)} \in \ P_{(k)}^{\cat}$.
By the semi-continuity hypothesis it follows that $f(\x_{\star}) \leq \lim_{k \in K} f(\x_{(k)}) \leq \lim_{k \in K} f(\y_{(k)}) \leq f(\y_{\star})$. 
\end{proof}

%

Next, when the extended poll is deployed, a neighbor subsequence $\{\y_{(k)} \}_{k \in K}$ has an associated subsequence $\{\z_{(k)}\}_{k \in K}$, called the \textit{endpoint subsequence}, containing endpoints with $\z_{(k)} \in \mathcal{E}_{(k)}\left(\y_{(k)}\right)$ as in \Cref{def:extended_poll_point}.
%
%
In \Cref{subfig:analysis_extendedpoll}, the different subsequences are illustrated on a similar example as the RG one in \Cref{subfig:extended_poll_example}, but with different level curves.
%
%
The refining subsequence $\{\x_{(k)}\}_{k \in K}$ converges in the red plane;
%
%
the neighbor and endpoints subsequences reside in the green plane consisting of a piecewise linear function in two parts.
%
%

Finally, the convergence analysis concludes with a discussion on the stronger local minima obtained by incorporating the extended poll.
%
%
 Assuming $\xi \geq 0$,  $\x_{\star} \in \Omega$ and $\y_{\star} \in \Omega\,\cap\,\mathcal{N}(\x_{\star}; m_{\star})$,
the results obtained for the component-wise mixed-variable local minimum can be strengthened by considering interactions between categorical and continuous variables:
\begin{itemize}
    \item if $f \left(\x_{\star} \right) < f\left( \y_{\star} \right)$, then by the semi-continuity hypothesis of~ \Cref{thm:neighbors}, there exists $\epsilon>0$ such that  $f \left( \x_{\star} \right) \leq f \left( \y \right)$ for any $(\y^{\cat}, \y^{\continuous}) \in \mathcal{N}(\x_{\star}^{\cat},m_{\star}) \times \mathcal{B}(\x_{\star}^{\continuous};\epsilon)$ with $\y = \left(\y^{\cat}, \y^{\integer}, \y^{\continuous} \right) \in \Omega$.
    %
    
    %
    
    \item otherwise, $f \left(\x_{\star} \right) = f\left( \y_{\star} \right)$, and it follows that $f\left(\x_{\star}\right)=f\left(\y_{\star}\right)=f\left(\z_{\star}\right)$, since $f(\x_{(k)}) \leq f(\z_{(k)}) \leq f(\y_{(k)}) \, \text{for all } k \in K$, by definition.
    Therefore, the extended poll must have been performed around $\y_{(k)}$ for infinitely many $k \in K$ sufficiently large~\cite{AACW09a}, which guarantees the interaction between the continuous and categorical variables as in the previous case.

\end{itemize}
This covers the [$\cat-\continuous$, int] minimum in \Cref{def:partially_local_minimum}.
%
%
Proceeding further, suppose that $\xi=\infty$, implying that the extended poll is always employed.
In that case, the same conclusion is drawn for the case $f(\x_*)=f(\y_*)$ above, and
the interaction between the categorical and integer variables is trivially gained since quantitative polls are performed infinitely many times.
%
This addresses the [$\cat-\continuous$, $\cat-\integer$] minimum from \Cref{def:mixed_local_minimum}.
%
%

Note that, if the set of refining directions is dense in the unit sphere at $\z_{\star}^{\continuous}$, then  \Cref{thm:convergence_catmads} can be adapted to the endpoint subsequence $\{\z_{(k)}\}_{k\in K}$.
This provides a necessary condition for the continuous variables of $\z_{\star}$ with generalized directional derivatives.

\section{Computational experiments}
\label{sec:computation_exp}

This section compares a simple instance of \catmads on analytical problems with openly available state-of-the-art solvers.
Implementation details of the instance are provided in \Cref{sec:implementation_details}.
%
%
The solvers used for benchmarking are discussed in \Cref{sec:solvers}.
Finally, \Cref{sec:data_profiles} details data profiles that measure the relative performance of the solvers.
%


%
%
%

%
The problems created for benchmarking this work are summarized in \Cref{tab:unconstrained_problems,tab:constrained_problems}.
These problems are either modifications of existing ones from the literature, or adaptations of classic continuous optimization problems, such as the Rosenbrock and Rastrigin problems.
Techniques employed to create mixed-variable problems include replacing continuous variables with functions depending on both continuous and categorical variables, and transforming a function into another one with cases, each assigned to a category.
%
The implementation of \catmads is publicly available at \url{https://github.com/bbopt/CatMADS_prototype}.
The complete presentation of the problems is contained in {\sf Cat-Suite}, a new collection of mixed-variable problems described in this technical report~\cite{G-2025-39}.
%

\setlength{\tabcolsep}{2pt}
\begin{table}[htb!]
    \footnotesize
    \centering
    \renewcommand{\arraystretch}{1.2}
    \begin{tabular}{lccccccl}
        \hline\hline
        Name & $n^{\cat}$ & $\ell$ & $n^{\integer}$ & $n^{\continuous}$ & Smooth & Ref. & Original problem \\
        \hline \hline
        Cat-1   & 2 & 9 & 2 & 4  & No  &\cite{jamil2013literature}  & Ackley  \\ \hline
        Cat-2   & 2 & 9  & 2  & 3  & No  &\cite{jamil2013literature}  & Beale \\ \hline
        Cat-3   & 2  & 4  & 2  & 2  & Yes  &\cite{PeBrBaTaGu2019}  & Augmented Branin  \\ \hline
        Cat-4    & 2  & 4  & 2  & 4  & No  &\cite{jamil2013literature}  & Bukin-6  \\ \hline
        Cat-5   &  1 & 6  & 1  & 3  & Yes  &\cite{LuVl00}  & EVD-52  \\ \hline
        Cat-6    & 2  & 9  & 0  & 2  & Yes  &\cite{PeBrBaTaGu2019} & Goldstein \\ \hline
        Cat-7  & 3  & 16  & 3  & 2  & No  &\cite{jamil2013literature} & Goldstein-Price  \\ \hline
        Cat-8  &  1 & 4  & 1  & 5  & Yes &\cite{LuVl00}  & HS78   \\ \hline
        Cat-9   & 2  & 9  & 2  & 8  & No  &\cite{jamil2013literature}  & Rastrigin  \\ \hline
        Cat-10  & 2  & 6  & 2  & 4  & No  &\cite{jamil2013literature} & Rosenbrock  \\ \hline
        Cat-11  & 1 & 4  & 1  & 4  & Yes  &\cite{LuVl00} & Rosen-Suzuki  \\ \hline
        Cat-12  & 1  & 5  & 5  & 5  & No  &\cite{jamil2013literature}  & Styblinski-Tang  \\ \hline
        Cat-13 & 1  & 10  & 0  & 4  & Yes  &\cite{MMZ2019}  & Toy   \\ \hline
        Cat-14  & 1 & 10  & 0  & 8  & No  &\cite{MMZ2019} & Toy   \\ \hline
        Cat-15  & 1  & 5  & 3  & 4  & Yes  &\cite{LuVl00}  & Wong-1  \\ \hline
        Cat-16  & 2  & 9  & 2  & 4  & No  &\cite{jamil2013literature}  & Zakharov  \\ \hline
    \end{tabular}
    \caption{Unconstrained test problems in {\sf Cat-Suite}.}
    \label{tab:unconstrained_problems}
\end{table}
%
\begin{table}[htb!]
    \footnotesize
    \centering
    \renewcommand{\arraystretch}{1.2}
    \begin{tabular}{lcccccccl}
        \hline\hline
        Name & $n^{\cat}$ & $\ell$ & $n^{\integer}$ & $n^{\continuous}$ & $m$ & Smooth & Ref. & Original problem \\
        \hline \hline
        Cat-17   & 2  & 9  & 2  & 3  & 3  & No  &\cite{jamil2013literature} & Beale \\ \hline
        Cat-18  & 2  & 4  & 2  & 2  & 1  & Yes  &\cite{PeBrBaTaGu2019} & Augmented Branin \\ \hline
        Cat-19  & 2  & 9 & 2  & 4  & 2  & No  &\cite{jamil2013literature}  & Bukin-6 \\ \hline
        Cat-20  & 1  & 4  & 4  & 4  & 3  & Yes  &\cite{LuVl00} & Dembo-5 \\ \hline
        Cat-21 & 1  & 6  & 1  & 3  & 1  & No &\cite{LuVl00}  & EVD-52 \\ \hline
        Cat-22  & 2 & 9  & 2  & 3  & 4  & Yes &\cite{CrBeOrSaSa2017}  & G-09  \\ \hline
        Cat-23   & 2  & 9  & 0  & 2  & 1  & Yes  &\cite{PeBrBaTaGu2019} & Goldstein  \\ \hline
        Cat-24  & 2  & 25  & 2  & 2  & 2  & Yes  &\cite{jamil2013literature}  & Himmelblau  \\ \hline
        Cat-25  & 2 & 4  & 3 & 5  & 4  & Yes &\cite{LuVl00}  & HS-114 \\ \hline
        Cat-26 & 1 & 3  & 2  & 4  & 6  & Yes  &\cite{LuVl00}  & Pentagon  \\ \hline
        Cat-27  &  1 & 8  & 2  & 2 & 3  & Yes  &\cite{CrBeOrSaSa2017} & Pressure-Vessel  \\ \hline
        Cat-28  & 2  & 25  & 1  & 2  & 2  & Yes  &\cite{CrBeOrSaSa2017} & Reinforced-Concrete-Beam  \\ \hline
        Cat-29  & 2 & 6 & 2 & 4 & 1 & No  &\cite{jamil2013literature}  & Rosenbrock  \\ \hline
        Cat-30  & 1 &  5 & 2 & 2  & 2  & No  &\cite{jamil2013literature}  & Styblinski–Tang  \\ \hline
        Cat-31  & 1 & 10 & 0 & 4 & 2  & Yes &\cite{MMZ2019}  & Toy \\ \hline
        Cat-32  & 1 & 6 & 4 & 6 & 3 & Yes &\cite{LuVl00}  & Wong-2  \\ \hline
    \end{tabular}
    \caption{Constrained test problems in {\sf Cat-Suite}.}
    \label{tab:constrained_problems}
\end{table}

\subsection{Implementation details}
\label{sec:implementation_details}

\begin{algorithm}[htb!]
\SetAlgoLined
Given $f:\mathcal{X} \to \overline{\mathbb{R}}$, a domain $\mathcal{X}$ and constraint functions $g_j:\mathcal{X} \to \overline{\mathbb{R}}$ with $j \in J$.

\SetKwBlock{init}{\textbf{0. Initialization}}{}
\init{
\setlength{\parskip}{0pt} 
\setlength{\parsep}{0pt}
\setlength{\topsep}{0pt}
\begin{tabbing}
$k \leftarrow 0$ \hspace{4.5cm} \= iteration counter\kill
%
%
$k \leftarrow 0$ \hspace{4cm} \> iteration counter\; \\
$h_{(0)}^{\text{max}} = \infty$ \hspace{4cm} \> initial barrier parameter\; \\
$m \in \mathbb{N}$ \hspace{4cm} \> number of points in categorical polls\; \\
%
$\bm{\Delta}_{(0)} \in \mathbb{R}_{+}^{n^{\quant}}$ \hspace{4cm} \> initial frame size parameters (\gmads )\; \\
$\bm{\delta}_{(0)} \in \mathbb{R}_{+}^{n^{\quant}}$ with $\bm{\delta}_{(0)}^{n^{\quant}} \leq \bm{\Delta}_{(0)}^{n^{\quant}}$ \hspace{1cm} \> initial delta size parameters (\gmads ) \; \\
$\mathbb{X} \leftarrow \{\x_{(1)}, \x_{(2)}, \ldots, \x_{(p)}\}$ \hspace{4cm} \> initial DoE of $p \geq 2$ points\;  \\
Set $\x_{\feasible}$ and $\x_{\infeasible}$  with~\eqref{eq:best_feasible_and_infeasible} \hspace{4cm} \> initial solutions (only one might exist)\; \\
$\xi \in \overline{\mathbb{R}}$ \>  parameter controlling $\xi$-extended poll triggers in~\eqref{eq:extended_poll_test}~and~\eqref{eq:selected_points}  \; \\
Construct $d^{\cat}$ \> problem-specific categorical distance\;
\end{tabbing}}

\SetKwFor{While}{While}{do}{end}
\While{no stopping criteria is met}{
    
    \SetKwBlock{search}{1. \textbf{Opportunistic search} $S_{(k)} = S_{(k)}^{\text{\normalfont SPC}} \cup S_{(k)}^{\text{\normalfont QUAD}}$}{}
    \search{
    \setlength{\parskip}{0pt} 
    \setlength{\parsep}{0pt}
    \setlength{\topsep}{0pt}
    
    
    
    \textbf{If} quantitative poll $P_{(k-1)}^{\quant}$ or speculative search $S_{(k-1)}^{\text{SPC}}$  is successful\;
     
        \hspace{0.5cm} 1.1. Speculative search $S_{(k)}^{\text{SPC}}$ on $\mathcal{X}^{\quant}$\;
    
    
    \vspace{0.01cm}

    1.2. Quadratic search $S_{(k)}^{\text{QUAD}}$ on $\mathcal{X}^{\quant}$\;
    \vspace{0.01cm}

    \vspace{0.15cm}
    }

    \SetKwBlock{poll}{2. \textbf{Opportunistic poll} $ P_{(k)} = P_{\text{\normalfont $\feasible$}}^{\cat} \cup P_{\text{\normalfont $\feasible$}}^{\quant} \cup P_{\text{\normalfont $\infeasible$}}^{\cat} \cup P_{\text{\normalfont $\infeasible$}}^{\integer} $}{}
    \poll{
    \setlength{\parskip}{0pt} 
    \setlength{\parsep}{0pt}
    \setlength{\topsep}{0pt}


    
    
    %
    }

    \SetKwBlock{extpoll}{3. \textbf{Opportunistic extended poll} $ \bar{P}_{(k)}$}{}
    \extpoll{
    \setlength{\parskip}{0pt} 
    \setlength{\parsep}{0pt}
    \setlength{\topsep}{0pt}


    
    
    %
    }

    \SetKwBlock{update}{4. \textbf{Update}}{}
    \update{
    \setlength{\parskip}{0pt} 
    \setlength{\parsep}{0pt}
    \setlength{\topsep}{0pt}
    4.1. Update incumbent solutions $\x_{\feasible}$ and $\x_{\infeasible}$ with~\eqref{eq:best_feasible_and_infeasible} \;
    \vspace{0.05cm}
    
    4.2. \gmads  and barrier update: \;
    %

    %
    

    
    %
    
    

    %
    
    

    \textbf{If} $\y \prec_f \x_{\feasible}$ or $\y \prec_h \x_{\infeasible}$ for some $\y \in S_{(k)} \cup P_{(k)}$ (dominating iteration)\;

    \hspace{0.5cm} Increase $\bm{\Delta_{(k+1)}}$ and $\bm{\delta_{(k+1)}}$, and set $h_{(k+1)}^{\text{max}}=h(\x_{\infeasible})$  \;
    %

    \textbf{Else if} $0<h(\y)<h(\x_{\infeasible})$ for some $\y \in S_{(k)} \cup P_{(k)}$ (improving)\;
    %
    
    \hspace{0.5cm} Set $h_{(k+1)}^{\text{max}}= \max \{ h(\bm{v})\,:\,h(\bm{v})<h(\x_{\infeasible}), \, \bm{v} \in \mathbb{X} \cup S_{(k)} \cup P_{(k)} \}$  \;

    \textbf{Otherwise} (unsuccessful)\;
    %
    
    \hspace{0.5cm} Decrease $\bm{\Delta}_{(k+1)}$ and $\bm{\delta}_{(k+1)}$, and set $h_{(k+1)}^{\text{max}}=h(\x_{\infeasible})$  \;
    %

    \vspace{0.15cm}
    }

    \SetKwBlock{stop}{5. \textbf{Termination}}{}
    \stop{
    \setlength{\parskip}{0pt} 
    \setlength{\parsep}{0pt}
    \setlength{\topsep}{0pt}
    \textbf{If} stopping criteria is met then \textbf{stop} \; 
    
    \textbf{Else} $\mathbb{X} \leftarrow \mathbb{X} \cup S_{(k)} \cup P_{(k)}$ and $k \leftarrow k+1$
    }
%
} 
\caption{A simple instance of \catmads based on cross-validation.}
\label{algo:catmads_prototype}
\end{algorithm}
%
The prototype instance of \catmads is presented in \Cref{algo:catmads_prototype} and it is intentionally simple.
%
The implementation is built on top of the \nomad software~\cite{nomad4paper}, a C++ implementation of the \mads  algorithm with many extensions such as \gmads.
The following steps are adapted from \nomad: the quantitative poll (\gmads), the initial DoE with a Latin hypercube sampling, the PB, the update, and the search steps.
%
%

%
%
%
%

\Cref{algo:catmads_prototype} starts by initializing the static and dynamic parameters.
The dynamic parameters are initially indexed by $k=0$.
The DoE must be performed during the initialization, since at least two points are required to construct the categorical distance $d^{\cat}:\mathcal{X}^{\cat} \times \mathcal{X}^{\cat} \to \mathbb{R}_+$.
This distance is a weighted Euclidean distance comparing categorical components that are transformed into binary vectors using \textit{one-hot encodings}~\cite{BaDiMoLeSa2023}.
The weights reflect the importance of each binary variable assigned to a specific category and allows to construct problem-specific categorical distances.
%
%
These weights are determined by minimizing the \textit{Root Mean Square Error} (RMSE) of an \textit{inverse distance weighting} interpolation.
The RMSE is evaluated using $3$-fold cross-validation on the DoE.
%
%
%
Integer and continuous variables have no weights, since these variables are not considered for the categorical distance.
%
%
%
In \Cref{sec:cat_mads}, recall that a mismatching-based distance on the RBG example yields $d^{\cat}(\text{R, B})=0=d^{\cat}(\text{R, G})$.
Now, with the proposed distance, the distances would be $d^{\cat}(\text{R,B})=\theta_{R}+\theta_{B}$ and $d^{\cat}(\text{R, G})= \theta_R+\theta_G$, where the real parameters $\theta_R$, $\theta_B$ and $\theta_G$ are tuned by regression.
This distance distinguishes categories more finely than mismatches.
In other words, the neighbors in distance-induced neighborhoods are selected in a meaningful manner.

The initial DoE is allocated 20\% of the total budget of evaluations.
This percentage was determined empirically through preliminary tests. 
It represents a considerable number of evaluations.
However, the categorical distance is constructed only once at the initialization, the DoE must be sufficiently large to ensure that the categorical distance adequately models the mixed-variable objective function $f$.
%

%
The stopping criteria includes a budget limit on the number of evaluations and a convergence criterion triggered when the mesh size parameters reach their lower bounds.
The trigger parameter $\xi$ of the extended poll dictates which type of local minimum is associated to the convergence condition:
if $\xi<0$, then the convergence condition represents a necessary condition for a \big[$\cat$, $\continuous$, $\integer$\big] minimum;
else if $\xi \in \mathbb{R}_+$, then it represents a \big[$\cat-\continuous$, $\integer$\big] minimum, and otherwise $\xi=\infty$ represents a \big[$\cat-\continuous$, $\cat-\integer$\big] minimum.
%
%
%
%

%
%

The search is performed opportunistically.
If the previous iteration had a successful quantitative poll (feasible or infeasible), then the speculative search of \nomad is executed on the quantitative space.
The speculative (or dynamic) search generates candidate points along the last successful direction of \gmads ~\cite{AuDe2006, Le09b}. 
To capitalize on promising directions, the speculative search is repeated as long as it remains successful.

The default quadratic search of \nomad~\cite{CoLed2011} is also used on the quantitative variables.
At each iteration, quadratic models for the objective and constraint functions are constructed around incumbent solutions, with categorical variables fixed.
For a given incumbent solution, the proposed point minimizes its corresponding model of $f$, while either satisfying or respecting the barrier of its model constraints, for feasible and infeasible incumbents, respectively. 
%
%
For more details on the speculative and quadratic searches, see~\cite{nomad4paper}.

%

%
The number of points $m$ in the categorical polls
is constant and common for the feasible and infeasible categorical polls.
For a given problem, $m$ is set by the formula $m=\max\left(2, \sqrt{\left| \mathcal{X}^{\cat}\right|} \right)$, which scales with the total number of categories $\smash{\left| \mathcal{X}^{\cat} \right| = \prod_{i=1}^{n^{\cat}} \ell_i}$.
The square root balances the scaling of the product. 
The formula ensures $m\geq2$, to avoid cycling between two categorical components.
Alternatives formulas were tested, notably linear ones, \textit{e.g.}, $m=0.5 \left|\mathcal{X}^{\cat} \right|$,
but none performed better than the square-root formula in preliminary tests.

For the poll, the evaluation order prioritizes points generated from quantitative polls over those from categorical polls. 
Moreover, points from quantitative polls are ranked using quadratic models as in~\cite{nomad4paper}, while points from categorical polls follow their generation order.

\subsection{Comparison solvers}
\label{sec:solvers}

For benchmarking, each optimization problem is attributed a budget of evaluation of $250n$, where $n=n^{\cat}+n^{\integer}+n^{\continuous}$ is the total number of variables.

\catmads is benchmarked with the following state-of-the-art solvers, used in various fields involving blackboxes with categorical variables~\cite{JiKu2023, PrTr2024, SrArAh2025}:
\begin{itemize}
    \item \href{https://pymoo.org/customization/mixed.html}{\sf Pymoo}, a Python-based suite of optimization solvers. 
    The solver used for benchmarking is the mixed-variable and single objective implementation of a genetic algorithm~\cite{BlankDeb2020};


    \item  \href{https://hub.optuna.org/tags/cma-es/}{\sf Optuna}, a covariance matrix adaptation evolution strategy (CMA-ES) for categorical variables~\cite{AkSaYaOhKo2020};

    \item \href{https://botorch.org/}{\sf Botorch}, a BO solver, that uses a Gower-based distance kernel for categorical variables and Gaussian kernel for quantitative variables~\cite{BaKaJiDaLeWiBa2020}. 
    %

\end{itemize}
{\sf Pymoo}~and {\sf Optuna}~are used off the-shelf, with only evaluation budget specified and launched using the same DoE as \catmads.

{\sf Botorch}~requires several user-defined choices.
The hyperparameters of the GPs are adjusted via a maximum likelihood procedure.
%
The acquisition function is the Log Constrained Expected Improvement from~\cite{AmDaErBaBa2023}, optimized using {\sf Botorch}'s built-in solver \texttt{acqf\_mixed} with a multi-start of size $3n$.
{\sf Botorch}~does not explicitly distinguish categorical and integer variables.
The integer variables are continuously relaxed, and rounded to their nearest integer: this is a user choice made for simplicity.
The initial DoE also comes from \nomad.


%
Running Mixed Rastrigin (unconstrained) with 1,000 evaluations takes more than 24 hours with a 11th generation Intel i7-11800H (2.30 GHz) CPU and a NVIDIA GeForce RTX 3060 GPU.
For {\sf Botorch}, fully evaluating the budget of 250$n$ evaluations is not suitable for benchmarking across several instances of the problems. 
To reduce computational time, {\sf Botorch}'s stopping criterion, based on the relative decrease of the acquisition function, is employed with default parameters.
%
%
%

\subsection{Data profiles}
\label{sec:data_profiles}

%
%

%
All solvers are subject to some randomness.
Hence, each problem is run with five random seeds to ensure fairness.
An instance $p \in \mathcal{P}$ refers to a problem with a specific seed, where $\mathcal{P}$ is the set of instances.
There is a total of 160 instances, half constrained and half unconstrained.

%
%
%
%
%
%

%
For a given instance, let $f_{\star}$ be the smallest feasible objective function value and $f_0$ be an initial feasible function value.
Let $\mathcal{S}$ be a set of solvers considered for benchmarking.
A solver $s \in \mathcal{S}$ is said to have $\tau$-solved an instance $p \in \mathcal{P}$ when it has produced a feasible incumbent solution $\x_{\feasible}\in \Omega$ with a reduction $f_0 - f(\x_{\feasible})$ that is sufficient
relatively to the smallest reduction $f_0 - f_{\star}$ obtained by any solver on this instance~\cite{MoWi2009}.
%
The convergence test
can be expressed as
%
%
\begin{equation}
f_0 - f(\x_{\feasible}) \geq (1-\tau) ( f_0 - f_{\star})
\end{equation}
%
where $\tau \in [0,1]$ is a given tolerance and $f_0$ is either
\begin{itemize}
    \item the smallest function value in the common DoE for unconstrained problems~\cite{G-2025-36},

    \item or the largest function value among the first feasible solution find the solvers otherwise~\cite{G-2025-36}.
    For the unconstrained problems, this approach allows to compare solvers on problems without feasible solutions in the DoE.
\end{itemize}

A data profile~\cite{MoWi2009} represents the portion of $\tau$-solved instances by a solver $s \in \mathcal{S}$, such that
\begin{equation}
    \text{data}_s\left( \kappa \right) \coloneq \frac{1}{\left| \mathcal{P} \right|} \left| \left\{ p \in \mathcal{P} \, : \, \frac{k_{p, s}}{n_p + 1}    \leq \kappa \right\} \right| \in [0,1]
\end{equation}
where $k_{p,s} \geq 0$ is the number of evaluations required for the solver $s \in \mathcal{S}$ to $\tau$-solve an instance $p \in \mathcal{P}$ and $n_p$ is the number of variables in the instance $p \in \mathcal{P}$.
In the following plots, the horizontal axis $\kappa$ represents groups of $(n_{p}+1)$ evaluations.

\subsubsection{Results for the extended poll}

Before comparing \catmads with the other solvers, $\xi$ is chosen with data profiles showcasing the relative performance of \catmads with different values of $\xi$.
These data profiles are done respectively on unconstrained and constrained instances in 
\Cref{subfig:dataprofiles_unconstrained_EP,subfig:dataprofiles_constrained_EP}.

%
%

%
\begin{figure}[htb!]
\centering
\begin{subfigure}{\textwidth}
\centering
\includegraphics[width=1\linewidth]{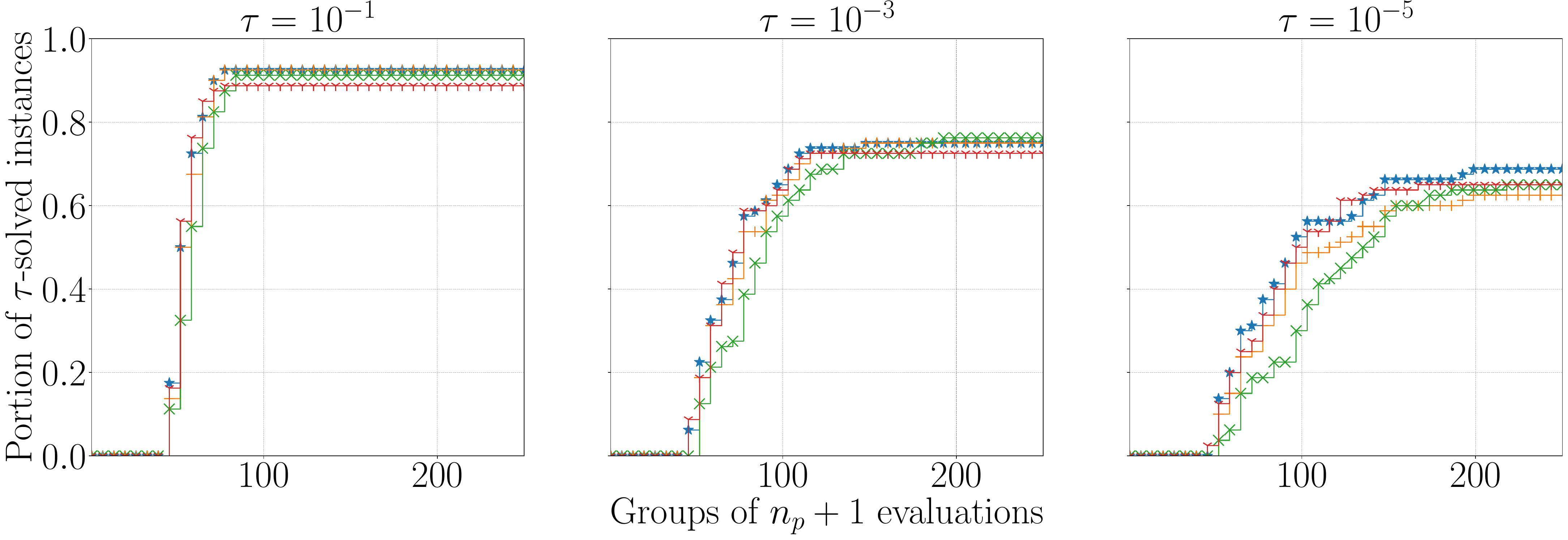}
\caption{Unconstrained test problems.}
\label{subfig:dataprofiles_unconstrained_EP}
\end{subfigure}
\begin{subfigure}{\textwidth}
\centering
\includegraphics[width=1\linewidth]{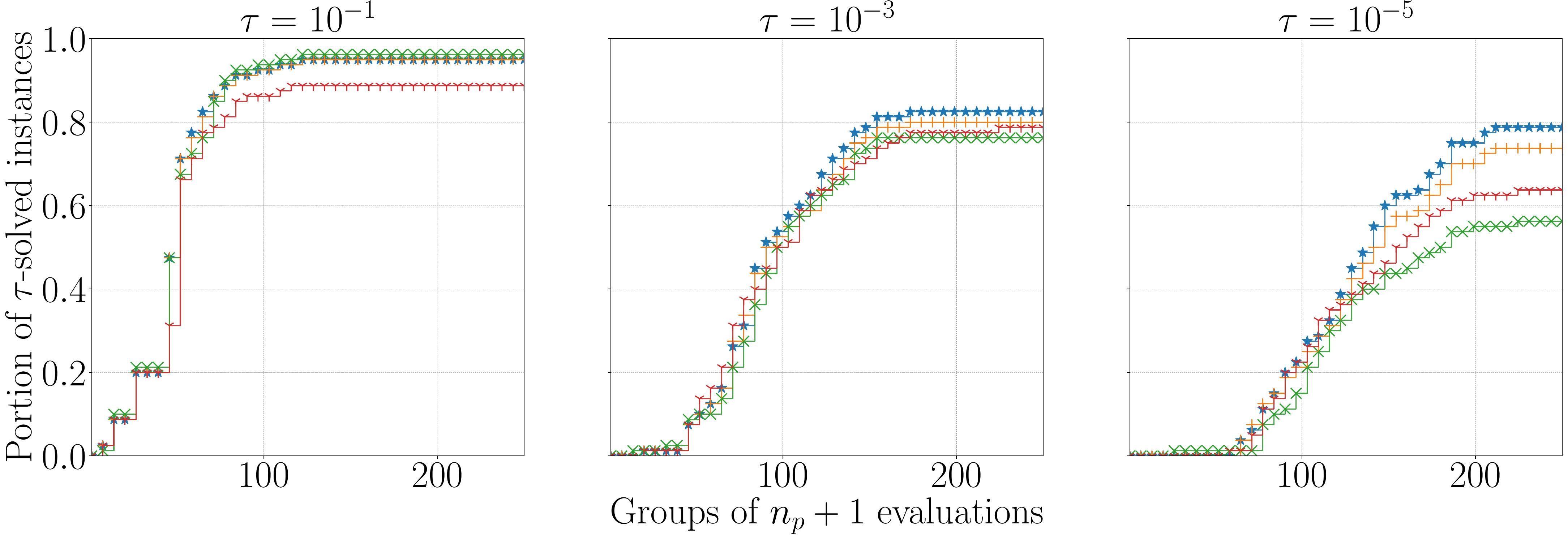}
\caption{Constrained test problems.}
\label{subfig:dataprofiles_constrained_EP}
\end{subfigure}
\begin{subfigure}{\textwidth}
\centering
\includegraphics[width=0.7\linewidth]{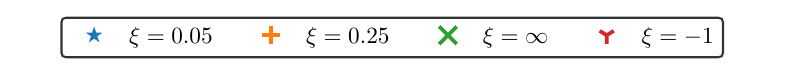}
\end{subfigure}
\caption{Data profiles for different values of $\xi$.}
\label{fig:data_profiles_EP}
\end{figure}

%
%
%
%
%
The value $0.05$, associated to a \big[$\cat-\continuous$, $\integer$\big] local minimum, yielded the best performance on the two smallest tolerances $10^{-3}$ and $10^{-5}$ for both unconstrained and constrained instances.
%
%
The value $\xi=\infty$ under-performs for the smallest tolerance $10^{-5}$.
%
Although $\xi=\infty$ is associated to the stronger \big[$\cat-\continuous$, $\cat-\integer$] local minimum its costs degrades performance. 
Based on these data profiles, the parameter $\xi$ is set to $0.05$.

\subsubsection{Results with existing solvers}

For the next data profiles, the solvers and \catmads are contained in the set of solvers 
$\mathcal{S}=\{\catmads, {\sf Pymoo}, {\sf Optuna}, {\sf BoTorch} \}$.
The data profiles comparing  \catmads with state-of-the-art solvers on the unconstrained and constrained instances are presented in \Cref{subfig:dataprofiles_unconstrained,subfig:dataprofiles_constrained}.
%
%
%
\begin{figure}[htb!]
\centering
\begin{subfigure}{\textwidth}
\centering
\includegraphics[width=1\linewidth]{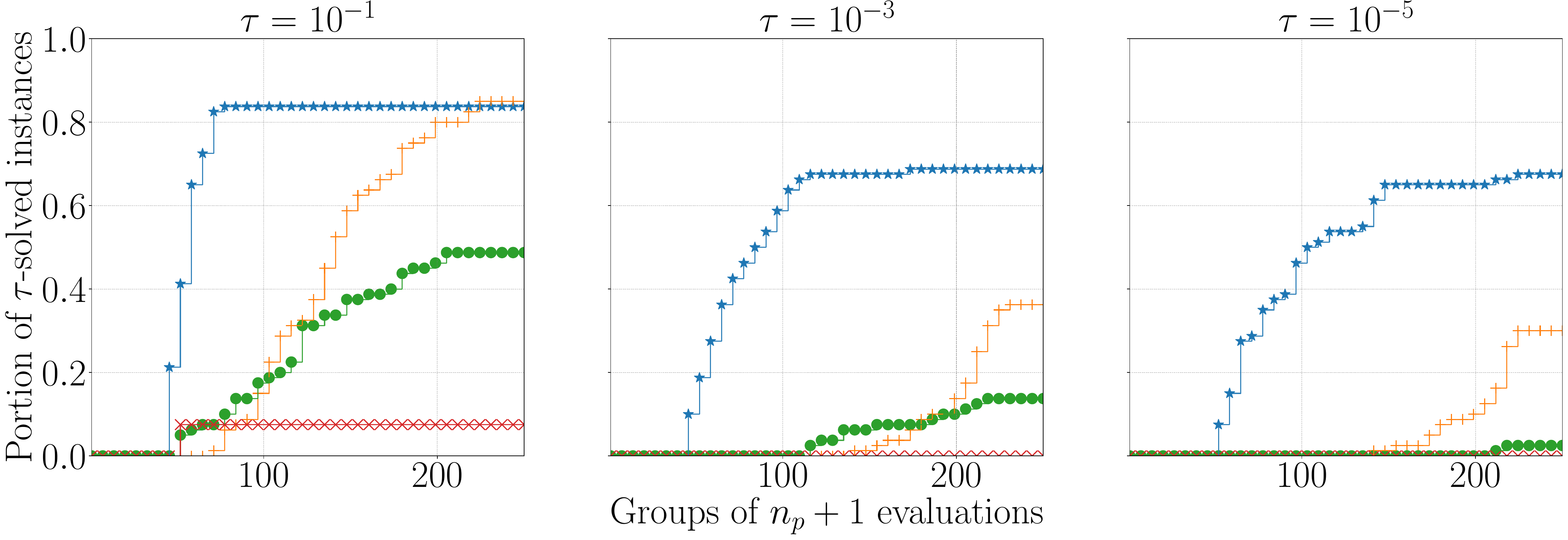}
\caption{Unconstrained test problems.}
\label{subfig:dataprofiles_unconstrained}
\end{subfigure}
\begin{subfigure}{\textwidth}
\centering
\includegraphics[width=1\linewidth]{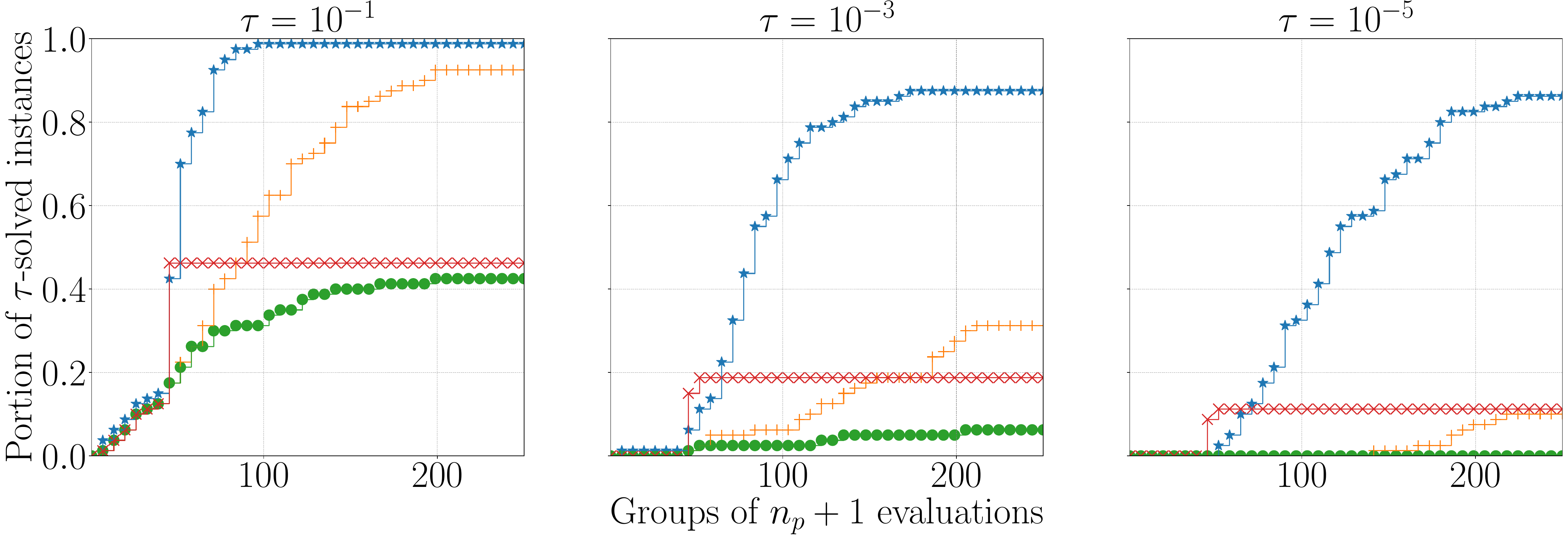}
\caption{Constrained test problems.}
\label{subfig:dataprofiles_constrained}
\end{subfigure}
\begin{subfigure}{\textwidth}
\centering
\includegraphics[width=0.7\linewidth]{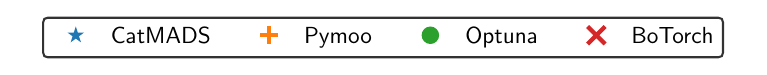}
\end{subfigure}
\caption{Comparison results on the test problems.}
\label{fig:data_profiles}
\end{figure}
The data profiles showcase that \catmads strongly dominates the other solvers.
%
For instance, the data profiles in \Cref{subfig:dataprofiles_unconstrained} with $\tau=10^{-5}$ shows that \catmads $\tau$-solves approximately 70\% of the problems against 30\% for {\sf Pymoo}.
%
%
\catmads performs particularly well on the smallest tolerance $\tau=10^{-5}$ and the constrained problems.
%

\section{Discussion}
\label{sec:discussion}

The present work introduces \catmads, a generalization of \mads  for mixed-variable constrained blackbox optimization.
\catmads improves the previous work \mvmads, in which discrete variables are treated with user-defined neighborhoods and constraints with the extreme barrier.
%
\catmads treats efficiently integer variables with the granular mesh and categorical variables with defined distance-based neighborhoods.
%
%
The constraints are handled with the progressive barrier~\cite{AuDe09a}, a much more efficient strategy than the extreme barrier~\cite{AuDe2006}.
The convergence analysis shows that \catmads has convergence guarantees under reasonable assumptions.
%

%
A prototype implementation of \catmads is compared to state-of-the-art solvers on 32 original mixed-variable problems.
%
%
Even if basic, the implementation is competitive and outperforms {\sf Pymoo}, {\sf Optuna} and {\sf BoTorch}. 
As such, \catmads addresses the research gap of an efficient mixed-variable constrained blackbox method with convergence guarantees.


The implementation in this work is intentionally simple.
%
%
In future work, \catmads will be hybridized with Bayesian optimization to produce an even more efficient solver.
%
Search steps focused on exploration will also be developed to balance the local nature of \catmads.
Moreover, \catmads will be applied on practical applications such as hyperparameter optimization of deep models or metaheuristic methods. 
%


\section*{Data availability statement}
Scripts and data are publicly available at \url{https://github.com/bbopt/CatMADS_prototype}.

\section*{Conflict of interest statement}
The authors state that there are no conflicts of interest.


%



\clearpage \newpage
\bibliographystyle{plain}
\bibliography{bibliography_catmads}
\pdfbookmark[1]{References}{sec-refs}

\end{document}